\documentclass[11pt,reqno]{amsart}

\usepackage[T1]{fontenc}
\usepackage[boldsans]{ccfonts}
\usepackage{eulervm}
\renewcommand{\mathbf}{\mathbold}
\usepackage{bbm}

\usepackage{enumerate, amsmath, amsthm, amsfonts, amssymb, 
mathrsfs, graphicx, paralist}
\usepackage[usenames, dvipsnames]{color}
\usepackage[margin=1in]{geometry} 
\usepackage{hyperref}
\usepackage{comment}

\usepackage{mathabx} 
\usepackage{tikz}
\usetikzlibrary{positioning,shapes,calc}
\usepackage{tikz-cd}

\numberwithin{equation}{section}
\newtheorem{Theorem}[equation]{Theorem}
\newtheorem{Proposition}[equation]{Proposition}
\newtheorem{Lemma}[equation]{Lemma}
\newtheorem{Corollary}[equation]{Corollary}
\newtheorem{Conjecture}[equation]{Conjecture}

\newtheorem{Assumption}[equation]{Assumption}

\theoremstyle{definition}
\newtheorem{Remark}[equation]{Remark}

\newtheorem{eg}[equation]{Example}

\newtheorem{Definition}[equation]{Definition}

\newcommand{\bA}{\mathbf{A}}

\newcommand{\bB}{\mathbf{B}}

\newcommand{\cC}{\mathcal{C}}

\newcommand{\sD}{\mathscr{D}}

\newcommand{\bG}{\mathbf{G}}

\newcommand{\cH}{\mathcal{H}}

\newcommand{\cI}{\mathcal{I}}

\newcommand{\sI}{\mathscr{I}}

\newcommand{\cK}{\mathcal{K}}

\newcommand{\bN}{\mathbf{N}}
\newcommand{\cN}{\mathcal{N}}

\newcommand{\cO}{\mathcal{O}}

\newcommand{\cT}{\mathcal{T}}

\newcommand{\bU}{\mathbf{U}}
\newcommand{\cU}{\mathcal{U}}

\newcommand{\cW}{\mathcal{W}}

\newcommand{\fc}{\mathfrak{c}}

\renewcommand{\AA}{\mathbb{A}}

\newcommand{\CC}{\mathbb{C}}

\newcommand{\GG}{\mathbb{G}}
\newcommand{\HH}{\mathbb{H}}

\newcommand{\QQ}{\mathbb{Q}}
\newcommand{\RR}{\mathbb{R}}

\newcommand{\ZZ}{\mathbb{Z}}

\renewcommand{\phi}{\varphi}
\renewcommand{\emptyset}{\varnothing}
\newcommand{\eps}{\varepsilon}

\renewcommand{\tilde}[1]{\widetilde{#1}}

\makeatletter
\def\Ddots{\mathinner{\mkern1mu\raise\p@
\vbox{\kern7\p@\hbox{.}}\mkern2mu
\raise4\p@\hbox{.}\mkern2mu\raise7\p@\hbox{.}\mkern1mu}}
\makeatother

\DeclareMathOperator{\sgn}{sgn}

\newcommand{\SL}{\mathbf{SL}}

\newcommand{\suchthat}{\mid}

\newcommand{\textif}{\text{ if }}
\newcommand{\textand}{\text{ }\mathrm{and}\text{ }}
\newcommand{\textor}{\text{ or }}
\newcommand{\Wv}{W^\mathrm{v}}

\newcommand{\kk}{\Bbbk}

\newcommand{\TitsCone}{\cT}
\newcommand{\WTits}{W_{\TitsCone}}

\begin{document}
\newcommand*\circled[1]{\tikz[baseline=(char.base)]{
            \node[shape=circle,draw,inner sep=2pt] (char) {#1};}}
\newcommand*\squared[1]{\tikz[baseline=(char.base)]{
            \node[shape=rectangle,draw,inner sep=2pt] (char) {#1};}}

\title{Double-affine Kazhdan-Lusztig polynomials via masures}
\author{Dinakar Muthiah}
\address{Kavli Institute for the Physics and Mathematics of the Universe (WPI), The University of Tokyo Institutes for Advanced Study, The University of Tokyo, Kashiwa, Chiba 277-8583, Japan}
\email{dinakar.muthiah@ipmu.jp}
\maketitle

\begin{abstract}

Masures (previously also known as hovels) are a generalization of the theory of affine buildings for arbitrary $p$-adic Kac-Moody groups. Gaussent and Rousseau invented masures to compute the Satake transform for $p$-adic Kac-Moody groups. Their answer is given as a sum over Hecke paths, which are certain piecewise linear paths.

Guided by their method we give a definition of double affine Kazhdan-Lusztig $R$-polynomials as a sum over piecewise linear paths that we call $I_\infty$-Hecke paths. Remarkably, the notion of $I_\infty$-Hecke path, which arises from masure theoretic considerations, is closely related to chains in the double affine Bruhat order. Our main result is that there are finitely many $I_\infty$-Hecke paths in untwisted affine ADE type. This implies that $R$-polynomials are well-defined in this case. This finiteness result follows from earlier known finiteness results for the double affine Bruhat order. Combined with other results on the double affine Bruhat order, we now have all the ingredients to define double affine Kazhdan-Lusztig $P$-polynomials.
\end{abstract}

\section{Introduction}

Let $\cK = \kk((\pi))$ be the Laurent series field over a finite field $\kk$, let $\cO = \kk[[\pi]]$ be its ring of integers, and let $\bG$ be a split reductive group. The group $G = \bG(\cK)$ is an an algebraic version of the $\bG$ loop group. One may view $G$ from a \emph{$p$-adic perspective} or a complementary \emph{geometric perspective}.

Under the $p$-adic perspective, $G$ is naturally a locally compact group and therefore has a Haar measure. One has the maximal compact subgroup $K=\bG(\cO)$ and the Iwahori subgroup $I$. Corresponding to these subgroups are the spherical Hecke algebra $\cH_K$, consisting of $K$-biinvariant functions on $G$, and the Iwahori-Hecke algebra $\cH_I$, consisting of $I$-biinvariant functions. Multiplication in these Hecke algebras is given by convolution with respect to the Haar measure.  The structure and representation theory of both algebras is well understood. We mention that this $p$-adic perspective makes equal sense when we replace $\cK$ by an arbitrary non-archimedean local field like $p$-adic numbers $\QQ_p$.

In contrast, the geometric perspective is limited to the case of $\cK = \kk((\pi))$. With this restriction one may consider $\bG(\cK)$ as an infinite-dimensional algebro-geometric object over $\kk$. Two quotients, the affine Grassmannian $\bG(\cK)/\bG(\cO)$ and the affine flag variety $\bG(\cK)/I$, have especially rich geometry. Perverse sheaves on these spaces and, in particular, the intersection cohomologies of Schubert subvarieties play a key role.

These perspectives complement one another and are linked precisely by Grothendieck's sheaf-function correspondence. For example, functions with values given by Kazhdan-Lusztig polynomials correspond to the intersection cohomology sheaves of Schubert varieties.

\subsection{The Kac-Moody case}

It is natural to extend the study of $G = \bG(\cK)$ to the case when $\bG$ is an infinite dimensional Kac-Moody group. The past decade has brought much progress in understanding $G = \bG(\kk((\pi)))$ from the $p$-adic perspective.

Many difficulties do arise. In the Kac-Moody case the group $G = \bG(\cK)$ is no longer locally compact, so there is no Haar measure. Therefore, even constructing the spherical Hecke algebra is a delicate issue. Nonetheless, Braverman and Kazhdan \cite{Braverman-Kazhdan} and independently Gaussent and Rousseau \cite{Gaussent-Rousseau-2014} have constructed the spherical Hecke algebra and established the Satake isomorphism.

Additionally, Braverman, Kazhdan, and Patnaik \cite{Braverman-Kazhdan-Patnaik} and independently Bardy-Panse, Gaussent, and Rousseau \cite{BardyPanse-Gaussent-Rousseau-2016} have constructed the Iwahori-Hecke algebra and established its Bernstein-Lusztig presentation. One subtlety in the Kac-Moody case is that the Cartan decomposition fails on the entire group, and therefore one must restrict attention to a subsemigroup $G^+ \subset G$ where it holds.  

The most interesting case is when $\bG$ is an untwisted affine Kac-Moody group. In this case, the Bernstein-Lusztig presentation establishes that the Iwahori-Hecke algebra is a slight modification of Cherednik's double affine Hecke algebra. We will mostly restrict attention to this untwisted affine case, and will therefore use the prefix ``double affine'' to describe objects in this setting.

\subsubsection{The geometric perspective}

Our primary interest is to further understand the second geometric perspective of Kac-Moody loop groups through a better understanding of double affine Schubert varieties. Presently however, very little is understood about Kac-Moody loop groups from the geometric point of view. 
In contrast, the $p$-adic perspective is better understood because the theorems one proves are natural generalizations of those in the reductive case. However, in the reductive case the geometric perspective crucially uses the coincidence that the loop group of a reductive group is itself a Kac-Moody group. For example this coincidence is responsible for the fact that affine Weyl groups are Coxeter groups. This coincidence is no longer available for Kac-Moody loop groups.

Nonetheless, there are many indications that a rich geometric theory is hidden in $G^+$. As a first approximation, we can consider the quotients $G^+/K$ and $G^+/I$ as the $\kk$-points of the double affine Grassmannian and the double affine flag variety. One promising sign is that $K$ and $I$ orbits on these spaces are indexed by discrete sets. This is directly analogous to the reductive case. We now explain this for the case of $I$-orbits on $G^+/I$.

Let $\Wv$ be the Weyl group of $\bG$, and let $\cT$ be the (integral) Tits cone. One may form the semi-direct product $\WTits = \cT \rtimes \Wv$. We write $\mu \in \cT$ multiplicatively as $\pi^\mu$, and for each $\pi^\mu w  \in \WTits $ we have a corresponding $I$-double coset $I \pi^\mu w I$. 
Braverman, Kazhdan, and Patnaik \cite{Braverman-Kazhdan-Patnaik} have established the following decomposition:
\begin{align}
  \label{eq:132}
 G^+ = \bigsqcup_{\pi^\mu w \in \WTits} I \pi^\mu w I 
\end{align}
In analogy with the reductive case, we call the subsets $I \pi^\mu w I/I \subset G^+/I$ \emph{double affine Schubert cells}. Our goal is to understand \emph{double affine Schubert varieties}, which should be the closures of $I \pi^\mu w I/I$ in $G^+/I$. In the reductive case, the closure is exactly controlled by the affine Bruhat order on the affine Weyl group. In the Kac-Moody case, $\WTits$ is not a Coxeter group, so there is a priori no obvious candidate for the closure order. Nonetheless, Braverman, Kazhdan, and Patnaik \cite{Braverman-Kazhdan-Patnaik} have given a definition of double affine Bruhat order that should play the role of the closure order.

As a first milestone in understanding double affine Schubert varieties, one would like to have a good understanding of double affine Kazhdan-Lusztig polynomials, which give their intersection cohomology stalks. Our goal in this paper is to establish a definition of these polynomials.

\subsubsection{Prior work}
In \cite{Muthiah-2018}, we established two results toward the above goal. First, we proved that the structure coefficients of the Iwahori-Hecke algebra are polynomials in $q = |\kk|$; this was independently obtained in \cite{BardyPanse-Gaussent-Rousseau-2016}. This polynomiality is essential for making sense of the substitution $q \mapsto q^{-1}$ that plays a central role in Kazhdan-Lusztig theory. Second, we proved a conjecture of Braverman, Kazhdan, and Patnaik by showing that the double affine Bruhat order on $\cT \rtimes \Wv$ is a partial order. In \cite{Muthiah-Orr-2019}, with Orr, we show that $\cT \rtimes \Wv$ has a well-behaved length function that controls the double affine Bruhat order. Finally, Welch has given a finer description of covers in the double affine Bruhat order in her PhD thesis \cite{Welch-2019}. These results play a key role in this paper.

\subsection{The present work}

The remaining ingredient for understanding Kazhdan-Lusztig polynomials is the Kazhdan-Lusztig involution. In the reductive case, this is exactly the data of the so called Kazhdan-Lusztig $R$-polynomials. It is known in the reductive case that $R$-polynomials are given by counting points in the intersections of Schubert cells and opposite Schubert cells. Let $I_\infty$ be the subgroup of $\bG(\kk[\pi^{-1}])$ consisting of elements that lie in the negative Borel $\bB^-(\kk)$ modulo $\pi^{-1}$. Opposite double affine Schubert cells are orbits of $I_\infty$ on $G^+/I$.
For each pair $\pi^\nu v, \pi^\mu w \in \WTits$, we are looking for the $R$-polynomial $R_{\pi^\nu v, \pi^\mu w} \in \ZZ[v]$ such that for any finite field $\kk$, we have
\begin{align}
  \label{eq:133}
 R_{\pi^\nu v, \pi^\mu w}(q)  =  |(I \pi^\mu w I \cap I_\infty \pi^\nu v I )/I|
\end{align}
where $q = |\kk|$.

\subsubsection{Recalling the Gaussent-Rousseau computation of the Satake transform}

Similar to establishing \eqref{eq:133} is the problem of computing
\begin{align}
  \label{eq:134}
  |(K \pi^\lambda K \cap U^-\pi^\nu  K )/K|
\end{align}
where $\lambda, \nu$ are dominant, and $U^- = \bU^-(\cK)$ where $\bU^-$ is the negative unipotent subgroup of the Kac-Moody group $\bG$. The numbers \eqref{eq:134} define the Satake transform. The problem of computing the Satake transform in the Kac-Moody setting was first investigated by Gaussent and Rousseau \cite{Gaussent-Rousseau-2008}. For their study, they developed the theory of \emph{masures} (previously also known as hovels). In the case when $\bG$ is reductive, masures recover the notion of affine buildings exactly. In the case where $\bG$ is of untwisted affine type, it is helpful to think of masures as playing the role of ``double affine buildings''. We note however that the failure of the Cartan decomposition means that masures cannot satisfy all the building axioms.

Gaussent and Rousseau reinterpret the set $K \pi^\lambda K /K$ as certain segments in the masure. Then they study the intersection of this set with $U^-\pi^\nu K /K$ via a retraction $\rho_{-\infty}$ to the standard apartment. The output of their approach is a decomposition of the set $(K \pi^\lambda K \cap U^-\pi^\nu  K )/K$ into finitely many pieces indexed by piecewise linear paths they call \emph{Hecke paths}. For clarity, we will call their notion \emph{$U^-$-Hecke paths} to distinguish from $I_\infty$-Hecke paths, which we will define in this paper.  They prove that the contribution of each $U^-$-Hecke path is given by a universal polynomial specialized at $q= |\kk|$. In particular, they give a purely combinatorial formula for $|(K \pi^\lambda K \cap U^-\pi^\nu K )/K|$. We mention that $U^-$-Hecke paths, while entirely combinatorial in nature, are quite complicated, and their combinatorics is still not fully understood.

\subsubsection{Our approach to masures}
Our goal is to imitate the above strategy in order to compute $R$-polynomials. Unfortunately, the currently available framework of masures is not on its own sufficient to solve this problem. Observe that the intersection \eqref{eq:133} only makes sense for equal characteristic fields like $\kk((\pi))$ and not for mixed characteristic fields like $\QQ_p$. However, the current axiomatic framework of masures cannot distinguish between these two cases.

So we proceed differently. Instead of using masure theoretic methods directly, we reinterpret the work of Gaussent and Rousseau in group theoretic terms. We define the notion of \emph{retraction along a subgroup} (see Proposition \ref{prop:unique-retraction-along-Q}), and then observe that $\rho_{-\infty}$ is exactly equal to $\rho_{U^-}$, which is the retraction  along the subgroup $U^-$ of $G$. In \S \ref{sec:the-Gaussent-Rousseau-computation-of-the-Satake-transform}, we review the work of Gaussent and Rousseau with an emphasis on group decompositions rather than axiomatic masure-theoretic considerations.

One of our key ideas is to reinterpret Gaussent and Rousseau's tangent building calculations in the framework of what we call \emph{Deodhar problems}. A Deodhar problem is the problem of computing the intersection of Schubert cell and a translate of an opposite Schubert cell in a Kac-Moody partial flag variety. There exists a general formula for the cardinality of such intersections originating in an algorithm first discovered by Deodhar \cite{Deodhar}.
We will see that each $U^-$-Hecke path specifies a finite number of Deodhar problems. The contribution of the $U^-$-Hecke path in the Gaussent-Rousseau formula is given by the product of the corresponding Deodhar problems.

In addition to guiding our approach to $R$-polynomials, we hope that our group-theoretic explanation of the work of Gaussent and Rousseau will make masures accessible to a wider audience. We also hope that our approach will also contribute to clarifying the algebro-geometric nature of masures. For example, it was observed by Contou-Carr\'ere \cite{Contou-Carrere} that galleries in affine buildings correspond to points in affine Bott-Samelson varieties. Presumably paths in masures should therefore correspond to points in double affine Bott-Samelson varieties.

\subsubsection{A masure theoretic definition of $R$-polynomials}

Guided by our understanding of the Gaussent-Rousseau masure-theoretic technique, we shall see that for computing $R$-polynomials, one needs to consider the retraction $\rho_{I_\infty}$ along $I_\infty$. Currently we do not construct $\rho_{I_\infty}$, which would involve establishing a generalized Birkhoff decomposition for $G^+$ and developing a theory of twinned masures. Constructing $\rho_{I_\infty}$ is an important problem that we will address in a future article with Manish Patnaik. In the current paper, we assume the existence of $\rho_{I_\infty}$ and obtain a combinatorial definition of $R$-polynomials. Whether our definition correctly computes \eqref{eq:133} is therefore conditional on $\rho_{I_\infty}$ existing and having formal properties analogous to those of $\rho_{U^-}$.
One can nonetheless use our combinatorial definition to develop the theory of double affine Kazhdan-Lusztig polynomials as we describe in \S \ref{sec:kazhdan-lusztig-polynomials}.

In the course of our computations using the conjecturally defined $\rho_{I_\infty}$, we quickly see the appearance of double affine roots and the double affine Bruhat order. We obtain a notion that we call \emph{$I_\infty$-Hecke paths}, which is intimately connected to the double affine Bruhat order. We then obtain formulas for $R$-polynomials as sum over $I_\infty$-Hecke paths. The contribution of each $I_\infty$-Hecke path is given by an explicit product of Deodhar problems in direct analogy with the case of the Satake transform.  In \S \ref{sec:computing-spherical-R-polynomials} we first consider spherical $R$-polynomials, whose computation is most analogous to the Satake transform calculation. This is the bulk of our work, and we obtain a definition of spherical $R$-polynomials (Definition \ref{def:spherical-R-poly}) as a result. Then in \S \ref{sec:iwahori-versions-of-R-polynomials} we consider Iwahori $R$-polynomials as a straightforward generalization (Definition \ref{def:iwahori-spherical-R-poly} and Definition \ref{def:iwahori-iwahori-R-poly}).

It is not a priori clear why only finitely many $I_\infty$-Hecke paths appear in the sum defining $R$-polynomials. Our main result (Theorem \ref{thm:finiteness-of-I-infty-Hecke-paths}) is that in untwisted affine ADE type the set of $I_\infty$-Hecke paths is a finite set. This crucially relies on our earlier work with Orr \cite{Muthiah-Orr-2019} and work of Welch \cite{Welch-2019} on the double affine Bruhat order. Therefore, we obtain a purely combinatorial definition of Kazhdan-Lusztig $R$-polynomials. It is apparent from this approach that the combinatorial complexity of our definition is closely related to the complexity of the double affine Bruhat order.

\subsubsection{Kazhdan-Lusztig $P$-polynomials}

Kazhdan-Lusztig $R$-polynomials are stepping stone to the real objects of interest: the Kazhdan-Lusztig $P$-polynomials that give the intersection cohomology stalks of Schubert varieties. As we explain in \S \ref{sec:the-iwahori-case}, we now have all the ingredients to run the original Kazhdan-Lusztig algorithm for computing $P$-polynomials : (1) double affine Bruhat order, (2) length function, and (3) $R$-polynomials. However, to get the positivity results of usual $P$-polynomials, we need to know that there is an affine algebraic variety whose intersection homology is given by the $P$-polynomials. We phrase the existence of this variety as Conjecture \ref{conj:existence-of-transversal-slices-to-schubert-varieties}. Geometrically, this affine algebraic variety should be thought of as a transversal slice to one double affine Schubert variety sitting inside another. 

In this approach to the $P$-polynomials, the Iwahori-Hecke algebra is not visible. In the finite and affine cases, it is known that the $R$-polynomials exactly give the matrix of the Kazhdan-Lusztig involution. In our double affine situation, our masure theoretic approach does not give us any such result. Nonetheless, we expect that if one appropriately completes the Iwahori-Hecke algebra for $G$, then the $R$-polynomials define the matrix of an algebra involution. We explain this briefly in \S \ref{sec:R-polynomials-and-the-matrix-of Kazhdan-Lusztig-involution}. We hope that such an interpretation will also clarify the complicated combinatorics of $I_\infty$-Hecke paths.

\subsubsection{Coulomb branches and affine Grassmannian slices}

In the spherical case, transversal slices to Schubert varieties are known as affine Grassmannian slices. There already exists a good candidate for double affine Grassmannian slices (and more generally Kac-Moody affine Grassmannian slices) due to the work of Braverman, Finkelberg, and Nakajima on Coulomb branches of $3d$ $\cN=4$ quiver gauge theories. In section
\ref{sec:affine-Grassmannian-slices-and-Coulomb-branches}, we briefly recall their work. We explain how this compares with our computation of a spherical $R$-polynomial in \S \ref{sec:an-sl-two-hat-example}. In particular, it appears that in the spherical case, one cannot capture the entire affine Grassmannian slice working only with minimal Kac-Moody groups. This also corresponds to some weird behavior of the double affine Bruhat order on the parabolic quotient $\WTits/\Wv$ (see \S \ref{sec:weird-phenomena-about-parabolic-quotients}). Additionally, it is known that Coulomb branches in Kac-Moody type have strange stratifications that do not appear in the reductive case. In some special cases, these strange stratifications are known to be related to various correction factors that arise in the study of $p$-adic Kac-Moody groups. A very interesting question is to understand double affine Grassmannian slices and their strange stratifications in terms of $p$-adic Kac-Moody groups.

\subsection{Acknowledgments}

I thank Alexander Braverman, Bogdan Ion, Hiraku Nakajima, Alexei Oblomkov, Daniel Orr, Manish Patnaik, Anna Pusk\'as, and Alex Weekes for helpful conversations. I am especially grateful to Nicole Bardy-Panse, St\'ephane Gaussent, Auguste H\'ebert, and Guy Rousseau for their generous explanations about the theory of masures. This work was supported by JSPS KAKENHI Grant Number JP19K14495.

\section{Recollections about masures}

Let $\bG$ be a split untwisted affine Kac-Moody group (with all the relevant auxiliary data of Borels, etc. fixed). Specifically, $\bG$ is the Tits functor corresponding to the ``minimal Kac-Moody group'' \cite{Tits}. Our discussion will apply to more general split Kac-Moody groups except at a few key points where we make use of our untwisted affine assumptions. We expect that these assumptions are not necessary.

Let $\bB$ and $\bB^-$ be the positive and negative Borel subgroups. Let $\bU$ and $\bU^-$ be the positive and negative unipotent subgroups. Let $\bA$ be the maximal torus. Recall that the standard apartment $\AA$ is the real-span of the cocharacter lattice $P$ of $\bA$. Let $\bN$ be the normalizer of the torus $\bA$. 

Let $\cK = \kk((\pi))$ be the field of Laurent series in a variable $\pi$ with $\kk$ coefficients. Let $\cO = \kk[[\pi]]$. Let $G = \bG(\cK)$ and $K = \bG(\cO)$. Let $B = \bB(\cK)$, $B^- = \bB^-(\cK)$, $U = \bU(\cK)$, and $U^- = \bU^-(\cK)$. Let $N = \bN(\cK)$ and $A_\cO = \bA(\cO)$.

The Iwahori subgroup $I = \left\{ g \in K \suchthat g \in \bB(\kk) \!\!\mod \pi \right\}$ plays the role of the double affine Borel subgroup. A central role will be played by the group $I_\infty = \left\{g \in \bG(\kk[\pi^{-1}]) \suchthat g \in \bB^-(\kk) \!\!\mod \pi^{-1} \right\}$.

Let $\Wv$ be the (vectorial) Weyl group $N/A$. Recall that $\Wv$ is a Coxeter group. We also form the quotient $N/A_\cO$ which we identify with $W_P = \Wv \ltimes P$.

Then $N$ acts on $\AA$ via its quotient $W_P$. For $w \pi^\mu \in W_P$, we have: 
\begin{align}
  \label{eq:025}
  w \pi^\mu . x = w(x) + w(\mu)
\end{align}
Notice that we have a plus sign here, which is opposite the convention in \cite[\S 3.1]{Gaussent-Rousseau-2008}. We choose this convention to reduce the many minus signs that would otherwise occur throughout our calculations.

\subsection{Fixators and the definition of the masure}

Associated to each filter $\Omega$ of subsets of $\AA$, Gaussent and Rousseau \cite{Gaussent-Rousseau-2008} define a ``parahoric'' subgroups $P_\Omega$ and  ``fixator'' subgroups $\widehat{P}_\Omega$, such that $P_\Omega \subseteq \widehat{P}_\Omega \subseteq G$. In particular, to each point $x \in \AA$, one has $\widehat{P}_x$. The masure $\sI$ is defined \cite[Definition 3.15]{Gaussent-Rousseau-2008} to be the quotient $G \times \AA / \sim$ where $\sim$ is the equivalence relation on $G \times \AA$ defined by \begin{align}
  \label{eq:2}
  (g,x) \sim (h,y)
\end{align}
when there exists $n \in N$ such that $y = n x$ and $g^{-1}hn \in \widehat{P}_x$.
Given $(g,x) \in G \times \AA$, we  write $[g,x]$ for the corresponding point in $\sI$.
We mention that the original definitions and theory from \cite{Gaussent-Rousseau-2008} includes the assumption of $\kk = \CC$, but Rousseau removes this restriction in \cite{Rousseau-2016} (see also \cite{Rousseau-2017}). 

\subsubsection{The group action and fixators}

One identifies $\AA$ as a subset of $\sI$ via the map $x \mapsto [1,x]$. The left action of $G$ on $G \times \AA$ descends to $\sI$. For each subset $S \subseteq \AA$, one can define $G_S$ to be the fixator $G$ acting on the subset $S \subset \sI$. For each filter $\Omega$ of subsets in $\AA$, one defines the fixator of $G_\Omega$ by $G_\Omega = \bigcup_{S \in \Omega} G_S$. It is not immediately clear whether $G_\Omega$ is equal to $\widehat{P_\Omega}$. This is true (as well are other good properties) when $\Omega$ has so called ``good fixator'' \cite[Definition 4.1]{Gaussent-Rousseau-2008}. Gaussent and Rousseau show that a wide variety of filters have ``good fixator''. For example, points $x \in \AA$ have good fixator, which includes the statement that $G_x = \widehat{P_x}$. In the untwisted affine case, we have an equality $G_0 = K$ \cite[\S 3.4 ]{Gaussent-Rousseau-2014}.

\subsubsection{Tits cone}

Let $P^{++}$ denote the cone of integral dominant coweights, and let $P^{++}_\RR$ denote the corresponding real cone. The \emph{real Tits cone} $\cT_{\RR} \subseteq \AA$ consists of the set of coweights that pair non-negatively with almost all positive roots. The \emph{real open Tits cone} is 
$\cT_{\RR} \subseteq \overset{\circ}{\cT_{\RR}}$ is the subset of the Tits cone consisting of coweights that pair strictly positively with almost all positive roots.
The \emph{(integral) Tits cone} is defined as $\cT = P \cap \cT_{\RR}$, and the \emph{open (integral) Tits cone} is defined as $\overset{\circ}{\cT} = P \cap \overset{\circ}{\cT_{\RR}}$.
In the untwisted affine case, which is our focus, $\overset{\circ}{\cT_{\RR}}$ and $\overset{\circ}{\cT}$ are exactly
the \emph{positive level} coweights.

\subsubsection{The preorder and $G^+$}

One defines a preorder on $\AA$ by declaring $x \leq y$ if $y-x \in \cT_\RR$. Given points $[g,x],[g,y] \in \sI$, one defines $[g,x] \leq [g,y]$ if $x \leq y$. This is independent of choice of $g$ (see \cite[\S 4.3 2)]{Gaussent-Rousseau-2008}). The positive part $\sI^+$ of the masure is defined to be the set of all points that are greater than or equal to $0 \in \sI$ under this preorder. One also defines $G^+$ to be the set of all $g \in G$ such that $g.0 \geq 0$.

\subsubsection{Notation for segment germs}
\label{sec:notation-for-segment-germs}

Given $x, y \in \sI$ with $x \leq y$, one can form the segment $[x,y]$ consisting of all points $z \in \sI$ such that $x \leq z \leq y$. Gaussent and Rousseau define the \emph{segment germ} $[x,y)$ to be the filter consisting of all open neighborhoods of $x$ in the segment $[x,y]$. Gaussent and Rousseau also use the notation for $[y,x)$ for the segment germ consisting of open neighborhoods of $y$ in $[x,y]$. We will prefer to call this segment germ $(y,x]$.

For our purposes, it will be clearer to adopt some further notation for segment germs.  Our segment germs will all lie in the standard apartment, so they will be of the form $[p,p+v)$ for a point $p \in \AA$ and a vector $v \in \cT_\RR$. Notice that $[p,p+tv) = [p,p+v)$ for any $t >0$. Therefore we will define $[p,p+\eps v) = [p,p+v)$. Similarly, we will write $(p-\eps v, p] = (p-v,p]$. This notation is suggestive of the idea that segment germs are akin to tangent vectors.

\subsection{Retraction along subgroups}

Let $Q$ be a subgroup of $G$. Let $\sI^\bullet \in \{ \sI, \sI^+ \}$ and consider correspondingly $G^\bullet \in \{ G, G^+ \}$. Let $\AA^\bullet = \AA \cap \sI^\bullet$. We say that a map $\rho_{Q}: \sI^\bullet \rightarrow \AA^\bullet$ is a \emph{retraction along $Q$} if $\rho_Q \vert \AA^\bullet : \AA^\bullet \rightarrow \AA^\bullet$ is the identity and $\rho_Q$ is $Q$-invariant.

We say that $G^\bullet$ \emph{admits generalized Iwasawa decomposition} with respect to $Q$ if for all $x \in X$
\begin{align}
  \label{eq:3}
 G^\bullet \subseteq Q N G_x
\end{align}
and further the $N$-factor of the above decomposition is unique up to right multiplication by $N \cap G_x$.

\begin{Proposition}
  \label{prop:unique-retraction-along-Q}
Suppose $G^\bullet$ admits generalized Iwasawa decomposition with respect to $Q$, then there is a unique retraction $\rho_Q$ along $Q$. 
\end{Proposition}

\begin{proof}
  Suppose $\rho_Q$ is a retraction along $Q$. Let $[g,x] \in \cI^\bullet$. We decompose $g = q n p$ according to \eqref{eq:3}. By $Q$-invariance, we must have $\rho_Q( [ g,x]) = \rho_Q( [n p, x])$. Because $p \in G_x$, we have $[n p,x] = [n,x]$. Finally, because $n \in N$, we have $n x \in \AA$, so $[n,x] = [1,nx]$. Therefore, we have
  \begin{align}
    \label{eq:4}
    \rho_Q( [ g,x]) = \rho_Q([1,nx]) = nx
  \end{align}
  where the second equality is determined by the retraction property of $\rho_Q$. Moreover, the uniqueness properties of the decomposition \eqref{eq:3} tell us that the value computed \eqref{eq:4} is independent of choice of decomposition.
\end{proof}

Gaussent and Rousseau show that $U^-$ admits generalized Iwasawa decomposition for the entire group $G$ (see \cite[Proposition 3.6]{Gaussent-Rousseau-2008}). Therefore, we have a unique retraction $\rho_{U^-} : \sI \rightarrow \AA$ along $U^-$. One can check that $\rho_{U^-}$ is exactly equal to the retraction $\rho_{-\infty}$ ``onto $\AA$ with center the anti-dominant segment-germ'' constructed by Gaussent and Rousseau \cite[Definition 4.5]{Gaussent-Rousseau-2008}.

Our description above clarifies the group theoretic nature of this retraction and gives a natural way to generalize retraction to groups not arising as fixators of filters. However, our description does not immediately give all the good properties of the retraction. For example, we do not immediately see the fact that $\rho_{-\infty}$ maps straight segments in the masure to piecewise linear paths in the standard apartment.

\subsection{A zoo of Weyl groups and root systems.}
\label{sec:weyl-groups-and-root-systems}

Let $\Phi^v$ be the root system for $\bG$.  Then we can define $\Phi^{v,+}$ and $\Phi^{v,-}$ to be cones of positive and negative roots respectively. Let us write $\beta > 0$ (resp. $\beta < 0$) if $\beta \in \Phi^{v,+}$ (resp. $\beta \in \Phi^{v,-}$).

Let $\Phi = \{ \beta + n \pi \suchthat \beta \in \Phi^v \textand n \in \ZZ\}$ be the $\bG$-affine root system.
Define 
\begin{align}
  \label{eq:35-imported}
\Phi^{+} =   \{ \beta + n \pi \suchthat n > 0 \textor  (n = 0 \textand \beta>0 ) \}
\end{align}
and:
\begin{align}
  \label{eq:36-imported}
\Phi^{-} =   \{ \beta + n \pi \suchthat n < 0 \textor  (n = 0 \textand \beta<0 ) \}
\end{align}
These two correspond to the real root spaces of $I$ and $I_\infty$ respectively.

Define affine:
\begin{align}
  \label{eq:33-imported}
\Phi^{U^+} =   \{ \beta + n \pi \suchthat \beta \in \Phi^{v,+} \textand n \in \ZZ \}
\end{align}
and:
\begin{align}
  \label{eq:34-imported}
\Phi^{U^-} =   \{ \beta + n \pi \suchthat \beta \in \Phi^{v,-} \textand n \in \ZZ \}
\end{align}
These two correspond to the real root spaces of $U^+$ and $U^-$ respectively.

\subsubsection{Notation for positive $\bG$-affine roots}

\newcommand{\betacheck}{\beta^\vee}

Let us define the signum function $\sgn : \ZZ \rightarrow \{\pm 1\}$ by
\begin{align}
  \label{eq:020-imported}
  \sgn(n) =
  \begin{cases}
    + 1 & \textif n \geq 0 \\
    - 1 & \textif n < 0. 
  \end{cases}
\end{align}
For $\beta \in \Phi^{v,+}$, we define
\begin{align}
  \beta[n] = \sgn(n)\cdot (\beta + n \pi) = \sgn(n)\beta + |n| \pi,
\end{align}
Then:
\begin{align}
  \label{eq:45-imported}
  \{ \beta[n] \suchthat \beta \in \Phi^{v,+} \textand n \in \ZZ \} = \Phi^+
\end{align}
For $\beta[n] \in \Phi^+$, we define
\begin{align}
  \label{eq:46-imported}
  s_{\beta[n]} = \pi^{n \betacheck} s_\beta
\end{align}

\subsubsection{One parameter-subgroups}

For each real root $\gamma \in \Phi^v$, we have the one parameter subgroup
\begin{align}
  \label{eq:57}
  e_\gamma : \GG_a \rightarrow \bG
\end{align}
which is a natural transformation of group functors from the additive group $\GG_a$ to the group $\bG$.

For each $\beta[n] \in \Phi^+$, we will define a corresponding one parameter subgroup
\begin{align}
  \label{eq:53}
 e_{\beta[n]} : \kk \rightarrow I
\end{align}
by:
\begin{align}
  \label{eq:129}
  e_{\beta[n]}(x) =
  \begin{cases}
    e_\beta(x \pi^n) & \textif n \geq 0 \\
    e_{-\beta}(x \pi^{-n}) & \textif n < 0 
  \end{cases}
\end{align}

Similarly, we define:
\begin{align}
  \label{eq:130}
 e_{-\beta[n]} : \kk \rightarrow I_\infty
\end{align}

\subsubsection{Root systems and Weyl groups associated to points}

Let $x \in \AA$. Then we can form
\begin{align}
  \label{eq:37-imported}
  \Phi^v_x = \{ \beta \in \Phi^v \suchthat \langle \beta, x \rangle \in \ZZ \}
\end{align}
and:
\begin{align}
  \label{eq:38-imported}
  \Phi^{v,+}_x = \Phi^{v,+} \cap \Phi^{v}_x  \\
  \Phi^{v,-}_x = \Phi^{v,-} \cap \Phi^{v}_x  
\end{align}
Then $\Phi^v_x$ is a Kac-Moody root system \cite{Bardy} with Weyl group $\Wv_x$.

We can also form
\begin{align}
  \label{eq:40-imported}
  \Phi_x = \{ \pm \beta[n] \in \Phi^+ \suchthat s_{\beta[n]} . x = x \}
\end{align}
and:
\begin{align}
  \label{eq:41-imported}
  \Phi^{U^+}_x = \Phi^{U^+} \cap \Phi_x  \\
  \Phi^{U^-}_x = \Phi^{U^-} \cap \Phi_x  
\end{align}
Then, as in the vectorial case, $\Phi_x$ is a Kac-Moody root system with Weyl group $W_x$.

Similarly, we define
\begin{align}
  \label{eq:57-imported}
  \Phi^{+}_x = \Phi^{+} \cap \Phi_x  \\
  \Phi^{-}_x = \Phi^{-} \cap \Phi_x  
\end{align}
For $x \in \overset{\circ}{\TitsCone}$, $\Phi_x$ is a Kac-Moody root system with Weyl group $W_x$ such that $\Phi^{+}_x$ is the set of positive real roots. Moreover, the subgroup $W_{[x,(1+\epsilon)x)}$ is a standard parabolic subgroup. This is immediate for $x$ dominant, and the general case follows by translating by an element of $\Wv$.

\subsubsection{Comparison with vectorial roots system and Weyl group}

We have the inclusion map $\Phi^v \hookrightarrow \Phi$ and the projection map $\Phi \twoheadrightarrow \Phi^v$.

For any $x \in \AA$, the induced map
\begin{align}
  \label{eq:42-imported}
\Phi_x \rightarrow \Phi^v  
\end{align}
is injective, and the image is exactly equal to $\Phi^{v}_x$.

Similarly, on the level of Weyl groups, we have the inclusion map $\Wv \hookrightarrow W$ and the projection map $W \twoheadrightarrow \Wv$.
Let us denote the projection map $W \twoheadrightarrow \Wv$ by:
\begin{align}
  \label{eq:54-imported}
  z \mapsto [z]
\end{align}

Also for any $x \in \AA$, the induced map
\begin{align}
  \label{eq:55-imported}
  W_x \rightarrow \Wv
\end{align}
is injective, and the image is exactly equal to $\Wv_x$.
\begin{Definition}
  Let us write
  \begin{align}
    \label{eq:47-imported}
   \iota_{x} : \Wv_x \rightarrow W_x 
  \end{align}
  for the inverse map to the isomorphism:
  \begin{align}
    \label{eq:56-imported}
    [\cdot]: W_x \rightarrow \Wv_x
  \end{align}
\end{Definition}

Let $z \in \Wv$ and $x \in \AA$. Then there is a unique $z^{(1)} \in \Wv_x$ such that:
\begin{align}
  \label{eq:39-imported}
  z ( \Phi^{v,-}) \cap \Phi^{v}_x = z^{(1)} (\Phi^{v,-}_x) 
\end{align}

Let:
\begin{align}
\Sigma^{v}_x = \{ v \in \Wv \suchthat v(\Phi^{v,-}) \cap \Phi^{v}_x = \Phi^{v,-}_x \}
\end{align}
Note that $\Sigma^{v}_x$ is only a set. It is not a subgroup of $\Wv_x$.
Then we have $ {(z^{(1)})}^{-1} \cdot z \in \Sigma^v_x$. 
\begin{Lemma}
  \label{lem:parabolic-factorization}
  As sets, $\Wv$ factors as a direct product:
  \begin{align}
    \Wv =  \Wv_x \cdot \Sigma^v_x
  \end{align}
  That is, for all $z \in \Wv$, we can uniquely factor
  \begin{align}
   z = z^{(1)} \cdot z^{(2)} 
  \end{align}
   where $z^{(1)} \in \Wv_x$ and $z^{(2)} \in \Sigma^v_x$.
\end{Lemma}

\begin{Definition}
  Let $x \in \AA$.
  Let us define 
  \begin{align}
   p_{x} :  \Wv \twoheadrightarrow \Wv_x
  \end{align}
 as follows. By Lemma \ref{lem:parabolic-factorization} we can uniquely
  \begin{align}
    z^{-1} = z^{(1)} \cdot z^{(2)}
  \end{align}
  where $z^{(1)} \in \Wv_x$ and $z^{(2)} \in \Sigma^v_x$. Then we define:
  \begin{align}
   p_{x}(z) = (z^{(1)})^{-1} 
  \end{align}
\end{Definition}

In other words, under the direct product factorization
\begin{align}
  \label{eq:052}
  \Wv = (\Sigma^v_x)^{-1} \cdot \Wv_x
\end{align}
$p_{x}$ sends an element of $\Wv$ to the right factor which lies in $\Wv_x$.

By \eqref{eq:39-imported}, we have:
\begin{align}
  z^{-1} ( \Phi^{v,-}) \cap \Phi^{v}_x = p_{x}(z)^{-1} (\Phi^{v,-}_x) 
\end{align}

\subsubsection{A lemma about Bruhat order on parabolic quotients}

Let $\lambda$  be a dominant weight, and let $\Wv_\lambda$ be the stabilizer of $\lambda$ under the action of $\Wv$. The group $\Wv_\lambda$ is a standard parabolic subgroup of $\Wv$, and therefore there is an induced Bruhat order on $\Wv/\Wv_\lambda$. Usually, one identifies $\Wv/\Wv_\lambda$ with the set of minimal representatives modulo $\Wv_\lambda$, and the Bruhat order is given by restricting the Bruhat order on $\Wv$ to the minimal length representatives. Here is an alternative description of the Bruhat order that does not explicitly mention minimal length representatives.

\begin{Lemma}
Let $x,y \in \Wv/\Wv_\lambda$. Then we have $x \geq y$ in $\Wv/\Wv_\lambda$ if and only if
there exists a sequence of positive roots $\beta_1, \cdots, \beta_s \in \Phi^{v,+}$, and elements
\begin{align}
  \label{eq:112}
 x.\lambda  = \xi_{0} , \xi_{1}, \cdots, \xi_{s} = y.\lambda  
\end{align}
such that for all $i$:
\begin{itemize}
\item  $\xi_i = s_{\beta_i}(\xi_{i-1})$
\item  $ \langle \beta_i, \xi_{i-1} \rangle < 0$
\end{itemize}
\end{Lemma}

\section{Deodhar problems and twinned buildings}
\label{sec:deodhar-problems-and-twinned-buildings}
\newcommand{\G}{\underline{G}}
\newcommand{\B}{\underline{B}}
\renewcommand{\P}{\underline{P}}
\newcommand{\W}{\underline{W}}

In this section, we will only over the finite field $\kk$. Let $\G$ be the $\kk$-points of a Kac-Moody group with opposite Borels $\B$ and $\B^-$, and Weyl group $\W$. Consider the flag variety $\G/\B$, and for $w,v \in \W$, we can consider the Schubert cell $\B w \B/\B$  and the opposite Schubert cell $\B^- v \B/\B$.
We will recall below a formula computing cardinality of the $\kk$-points of the intersection of Schubert cells with (translates) of opposite Schubert cells. We will the problem of computing such a cardinality a \emph{Deodhar problem} after Deodhar's work in \cite{Deodhar}. 

For our purposes in the rest of the article, $\G$ will be the group controlling the tangent building of a point in the masure. We will need to consider the situation that the Cartan matrix defining $\G$ is infinite: this is no difficulty because when we study Schubert varieties, our attention will always involve a finite submatrix of the Cartan matrix.

\subsection{Deodhar decomposition}

Let $w,v \in \W$. Deodhar in \cite{Deodhar} has given an explicit decomposition of $\B w \B/\B \cap \B^- v \B/\B$ in. His decomposition depends on a reduced word $\sigma = (\sigma_1,\ldots,\sigma_\ell)$ of $w$. He calls a subexpression $\tau = (\tau_1,\ldots,\tau_\ell)$ \emph{distinguished} if it satisfies a certain explicit combinatorial condition that we will omit here (\cite[Definition 2.3]{Deodhar} and see also \cite[Definition 4.1]{Muthiah-Orr-2018}). Write $\sD_\sigma$ for the set of distinguised subexpressions of $\sigma$. For each $\tau \in \sD_\sigma$, Deodhar constructs a subset $S_\tau$ of $\B w \B/\B \cap \B^- v \B/\B$ that is isomorphic to $\kk^a \times (\kk^\times)^b$ for some $a,b \geq 0$. Then he proves that:
\begin{align}
  \label{eq:1}
  \B w \B/\B \cap \B^- v \B/\B = \bigsqcup_{\tau \in \sD_\sigma} S_\tau
\end{align}
In particular, this gives an explicit formula for the cardinality of $\B w \B/\B \cap \B^- v \B/\B$.

\subsubsection{Attracting sets in Bott-Samelson varieties}

Let us briefly sketch a proof of the Deodhar decomposition using the geometry of Bott-Samelson varieties (see \cite[\S 5]{Muthiah-Orr-2018} for the argument in detail). Recall that $\B^- v \B/\B$ is precisely the attracting set for an explicit $\GG_m$-action on $\G/\B$ (coming from the cocharacter $\rho : \GG_m \rightarrow \G$). Corresponding to the reduced word $\sigma$, we can construct a Bott-Samelson variety $\Sigma(\sigma)$, which resolves the Schubert variety $\overline{\B w \B/\B}$. Specifically, we have an open set $\overset{\circ}{\Sigma(\sigma)}$ that maps isomorphically to the Schubert cell $\B w \B/\B$. The $\GG_m$ action $\rho$ lifts to to $\Sigma(\sigma)$. The torus fixed points on $ \Sigma(\sigma)$ are precisely indexed by the subexpressions $\tau$ of $\sigma$. The distinguished subexpression are precisely those subexpressions whose attracting sets have non-trivial intersection with $\overset{\circ}{\Sigma(\sigma)}$, and $S_\tau$ is precisely the intersection of $\overset{\circ}{\Sigma(\sigma)}$ with the attracting set of $\tau$. Because the Bott-Samelson variety comes equipped with an explicit atlas, we can compute $S_\tau$ explicitly and see that it is isomorphic to $\AA^a \times \GG_m^b$.

\subsubsection{The parabolic generalization}

Let $\P$ be a standard parabolic of $\G$. The proof also shows that it is easy to compute the cardinality of sets of the form $u \B w \P/\P \cap t \B^- v \P/\P$. We immediately see that this set is isomorphic to $\B w \P/\P \cap u^{-1}t \B^- v \P/ \P$. Repeating the argument, now instead using a cocharacter $\eta: \GG_m \rightarrow \G$ to produce a $\GG_m$-action having $u^{-1}t \B^- v \P/\P$ as an attracting set, we can compute a decomposition similar \eqref{eq:1}. The notion of distinguished expression will change according to $\eta$, but the overall computation will be similar. We call the problem of computing the cardinality of sets of the form $u \B w \P/\P \cap t \B^- v \P/\P$ a \emph{(parabolic) Deodhar problem}, and we see that parabolic Deodhar problems have explicit combinatorial solutions.

\begin{Remark}
  We sketched the proof above because our strategy for solving Deodhar problems for $p$-adic Kac-Moody groups
is a generalization of this method. Gaussent and Littelmann explain in \cite{Gaussent-Littelmann} how limits of various $\GG_m$-actions on Bott-Samelson varieties correspond to various retractions from buildings to their standard apartments. Because we do not have Bott-Samelson varieties for $p$-adic Kac-Moody groups, we proceed using masures.
\end{Remark}

\subsection{Twinned buildings}

\newcommand{\twinI}{\underline{\sI}}

One has, associated to $\G$, a twinned paired of buildings $(\twinI^-,\twinI^+)$. The chambers, $\cC^+$ of $\twinI^+$ are in bijection with $\G/\B$, and the chambers, $\cC^-$ are in bijection with $\G/\B^+$. The data of the twinning is a codistance:
\begin{align}
  \label{eq:5}
  d^* : \cC^- \times \cC^+ \rightarrow \W
\end{align}
In our case, the codistance can be described very explicitly in group-theoretic terms. Given $(g_1 \B^-,g_2 \B) \in \cC^- \times \cC^+$, we have $d^*((g_1 \B^-,g_2 \B)) = w$ where $w \in \W$ is uniquely determined by $\B^- g_1^{-1} g_2 \B = \B^- w \B$.

There is also a usual distance function $d : \cC^+ \times \cC^+ \rightarrow W$ with an analogous group-theoretic characterization. Using this description, one sees the following well known fact.
\begin{Proposition}
 The set $\B^- v \B/ \B \cap \B w \B/ \B $ is in bijection with the chambers $g\B \in \cC^+$ of distance $w$ from $\B$ and codistance $v$ from $\B^-$.
\end{Proposition}
We see that the the Deodhar problem can be phrased as a question of counting chambers in the building $\twinI^+$ with constrained distance and codimension. More general and parabolic Deodhar problems can be similarly phrased as counting certain faces in the building with constrained distance and codimension.

\subsubsection{Tangent buildings}

A key feature of the masure $\sI$ is that at every point $x \in \sI$, one has a twinned pair of \emph{tangent buildings} $(\sI^-_x, \sI^+_x)$. As a set $\sI^+_x$ consists of a segment germs $[x,y)$ where $y \in \sI$ and $x \leq y$. Similarly, $\sI^-_x$ consists of segment germs $(y,x]$ where $y \leq x$. Moreover, when $x \in \AA$ this building is modelled on the root system $\Phi_x$. We refer to \cite[\S 5.11 Proposition 4]{Rousseau-2017} for the precise statements. We remark that tangent buildings are sometimes called ``residue buildings'' in the literature.

\subsubsection{Buildings and Bott-Samelson varieties}
The usual method to solve Deodhar problems in the building theoretic context above is to consider all galleries of a fixed ``type'' starting at the fundamental chamber. The ``type'' corresponds to a reduced word of $w$. Then one computes which galleries end in a chamber of the prescribed codistance. In the process one produces a sum over folded galleries. The folded galleries precisely correspond to the distinguished expressions above.

The equivalence between this gallery-theoretic approach and the Bott-Samelson approach sketched above is originally due to Contou-Carr\'ere in his thesis \cite{Contou-Carrere}. We learned of this method, in particular the idea of thinking about retraction as limiting under a $\GG_m$-action,  in the work of Gaussent and Littelmann \cite{Gaussent-Littelmann}.

\subsubsection{Our purposes}

Our purpose of reviewing this correspondence between buildings and Bott-Samelson varieties is two fold. First, our overall idea is to rephrase our Deodhar problems for the $p$-adic Kac-Moody group, ~\eqref{eq:131},~\eqref{eq:128}, and \eqref{eq:72}, in masure theoretic terms. Unfortunately, because the theory of twinned masures has not yet been developed, we cannot conclude that we have solved the problem. However, we will be able to combinatorially compute the answer nonetheless.

Second, in the process of computing the answer, we will need to perform computations in the tangent buildings of various points in the masure. These tangent buildings are a twinned pair, and the computation will involve counting faces with fixed distance and codistance conditions. The existing calculations of this sort in tangent buildings involve the gallery-theoretic approach. Instead, we will rephrase the computation as precise Deodhar problems as discussed above. This way we need not unpack the combinatorics of solving these problems. This difference is largely cosmetic because the answers at the end will be identical. Nevertheless, we find it simplifying to encapsulate the complexity these tangent building computations and also clarifying for our purposes of generalization.

\section{The Gaussent-Rousseau computation of the Satake transform}
\label{sec:the-Gaussent-Rousseau-computation-of-the-Satake-transform}

In this section, we will review the Gaussent-Rousseau computation of the Satake transform \cite{Gaussent-Rousseau-2008}. Our goal is to present their work in a way that (1) emphasizes the role of Deodhar problems in the tangent building computations and (2) allows a straight-forward generalization to the case of $I_\infty$.

\subsection{The Satake transform and point counting}

Let $\lambda$ be a dominant coweight, and let $\nu$ be a coweight with $\nu \leq \lambda$. We want to understand the set
\begin{align}
  \label{eq:6}
  ( K \pi^\lambda K \cap U^- \pi^\nu K )/K
\end{align}
whose cardinality defines that Satake transform.

Gaussent and Rousseau have given an explicit formula for this cardinality given by counting paths in the masure. We will review their method, and explain their computation in the language we have explained above. Instead of using their masure-theoretic characterization of retraction, we use the group theoretic characterization of retraction $\rho_{U^-}$. Second, we will encapsulate the tangent building computations as Deodhar problems instead of unpacking the computations involving galleries. Our purpose is to present this computation in a way that will generalize to computing retraction along $I_\infty$.

\subsection{Identifying $ K \pi^\lambda K/K$ with a space of segments}

Let us consider the point $\lambda = \pi^{\lambda}.0$ (N.B. Gaussent and Rousseau would have a minus sign here). Then $\lambda \geq 0$, so we can form the straight line path
$\phi : [0,1] \rightarrow \AA$ given by:
\begin{align}
  \label{eq:07}
  \phi_0(t) = t \cdot (\lambda) = t \pi^\lambda.0
\end{align}

One can consider all $G$-translates of $\phi_0$, which  consists of paths of the form $t \mapsto [g, t \pi^\lambda.0]$ in $\cI$ for fixed $g \in G$. We also want this path to take the value $0$ at time $0$, which forces $g \in G_0=K$.

Therefore, we are considering all $K$ translates of the path $\phi_0$. We can see that the set of all end points of these paths are given by $K \pi^\lambda.0$, which we identify with $K \pi^\lambda K /K$. Moreover, because these paths are all segments and segments are determined by their endpoints, we identify $K \pi^\lambda K /K$ with the set of all $K$-translates of the path $\phi_0$. Let us write $K.\phi_0$ for the set of $K$-translates of $\phi_0$.

Observe that the stabilizer of $K$ acting on the path $\phi_0$ is precisely the fixator $G_{[0,\lambda]}$ of the segment $[0,\lambda]$. By \cite[\S 4.2 Example 1]{Gaussent-Rousseau-2008}, we have $G_{[0,\lambda]} = G_{\{0,\lambda \}} = G_{0} \cap G_\lambda = K \cap \pi^{\lambda} K \pi^{-\lambda}$. Observe that $K \cap \pi^{\lambda} K \pi^{-\lambda}$ is also the stabilizer  of $K$ acting on $\pi^\lambda K$ in $G/K$. Therefore, we have an $K$-equivariant isomorphism:
\begin{align}
  \label{eq:7}
  K . \phi_0 \cong   K \pi^\lambda K /K
\end{align}

\subsection{Computing retractions $\rho_{-\infty}$}
\label{sec:computing-retractions-rho-infty}

Let $\phi$ be a path in $K.\phi_0$. We can consider the pointwise retraction $\tau = \rho_{U^-}(\phi)$ of the path $\phi$. From our group theoretic description of $\rho_{U^-}$, we do not a priori see why $\tau$ is continuous. But from the masure-theoretic description of the retraction, one has that $\tau : [0,1] \rightarrow \AA$ is a piecewise linear path \cite[\S 4.4]{Gaussent-Rousseau-2008}.

Recall that $\phi(t) = [g, t \pi^\lambda.0]$ for some $g \in K$. Let us try to compute $\tau$ for very small $t$.
As in we explained \S \ref{sec:notation-for-segment-germs}, we write $[0,\epsilon\cdot\lambda)$ for the segment-germ consisting open neighborhoods of $0$ in $[0,\lambda]$. Therefore, we can consider the fixator subgroup $G_{[0,\epsilon\cdot\lambda)}$. To declutter notation let us drop the $\lambda$'s from our future notation (i.e. we will identify a point $t \in [0,1]$ with the point $t.\lambda \in \AA$. For example, we will simply write $[0,\epsilon)$ and $G_{[0,\epsilon)}$).

\subsubsection{The first step}
Because $[0,\epsilon)$ is a ``narrow'' filter, we have an Iwasawa decomposition 
\begin{align}
  \label{eq:9}
 G = U^- N G_{[0,\epsilon)} 
\end{align}
Let us write 
\begin{align}
  \label{eq:10}
 g = u_0 z_0 p_0 
\end{align}
where $u_0 \in U^-$, $z_0 \in W$, $p_0 \in G_{[0,\epsilon)}$.

Gaussent and Rousseau prove that $G_{[0,\epsilon)}$ is a good fixator. Among other properties we have
\begin{align}
  \label{eq:11}
  G_{[0,\epsilon)} = \bigcup_{t > 0} G_{[0,t]}
\end{align}
where again we abbreviate the notation for the segment $[0,t.\lambda]$ to $[0,t]$. Let $t_1$ be the maximal $t_1$ such that $g_0 \in G_{[0,t_1]}$. Choosing $t_1$ (and later $t_k$) like this will ensure that the description of $\tau$ is irredundant. Such a $t_1$ exists by the good properties of the retraction $\rho_{U^-}$; it is exactly the first folding time of the piecewise linear path $\tau$. Therefore, we see that:
\begin{align}
  \label{eq:12}
 \tau(t) = z_0.(t \lambda) \text{ for } t \in [0,t_1]
\end{align}

\subsubsection{Repeating the process}

Let us know try to compute $\tau$ for times greater than $t_1$. Let us write $[t_1,t_1+\epsilon)$ for the filter consisting of segments $\pi([t_1,t])$ for times $t$ slightly after $t_1$. The filter $[t_1,t_1+\epsilon)$ is narrow, so there is a corresponding Iwasawa decomposition. Let us perform the Iwasawa decomposition of $p_0$ with respect to $G_{[t_1,t_1+\epsilon)}$ on the right and $z_0^{-1} U^- z_0$ on the left, i.e. write
\begin{align}
  \label{eq:16}
  p_0 = z_0^{-1} u_1 z_1 p_1
\end{align}
where $u_1 \in U^-$, $z_1 \in W_P$, and $p_1 \in G_{[t_1,t_1+\epsilon)}$.

\newcommand{\bigcell}{\cU}

Next, we choose $t_2$ be the maximal $t_2$, such that $p_1 \in G_{[t_1,t_2]}$ and repeat the above discussion. We repeat the process until for some $N$ we have $1 \in [t_{N},t_{N+1}]$; redefine $t_{N+1}=1$. The fact that such an $N$ exists again uses the good properties of the retraction proved by Gaussent and Rousseau \cite[\S 4.4]{Gaussent-Rousseau-2008}. For all $k \leq N$, we have computed
 \begin{align}
  \label{eq:15}
   g = u_0 \cdots u_{k} z_{k} p_k
 \end{align}
 where $p_k \in G_{[t_k,t_{k+1}]}$. Therefore,
 \begin{align}
   \label{eq:17}
 \tau(t) = z_k. t\lambda \text{ for } t \in [t_k,t_{k+1}]
 \end{align}
that is, we have computed the retraction $\tau$. We see it is piecewise linear.

For $k =N$, we have:
\begin{align}
  \label{eq:18}
  g = u_0 \cdots u_{N+1} z_{N} p_{N+1}
\end{align}
 By construction $p_{N} .  \lambda =  \lambda$. So
\begin{align}
  \label{eq:20}
 g \pi^\lambda.0 =  u_0 \cdots u_{N} z_{N} . 0
\end{align}
So we must have $z_{N}.0 = \nu = \pi^\nu.0$.

\subsection{The space of parameters}
\label{sec:the-space-of-parameters}

Fix a retracted path $\tau$ as above, i.e. fixing the finite data of the folding times $(0=t_0 < t_1 < \cdots t_N < t_{N+1} = 1)$ and folding directions $(z_0,\cdots,z_N)$. We can then ask \emph{how many} paths in $K.\phi_0$ retract to $\tau$. Gaussent and Rousseau   prove that the space of of such parameters is a quasi-affine variety given as an explicit open subset of an affine space \cite[Theorem 6.3]{Gaussent-Rousseau-2008}. Moreover, this quasi-affine variety can be decomposed into pieces that are isomorphic to products of affine lines and multiplicative groups. One therefore concludes that the number of paths in $K.\phi_0$ that retract onto $\tau$ is given by a universal polynomial in $q = \# \kk$. We explain their calculations.

For each $i \in \{0,\ldots,N\}$, the decompositions \eqref{eq:15} for $k \leq i$ show that the data of $u_0,\ldots,u_i$ and $z_0,\ldots,z_i$ determine the path $\phi$ for times $t \in [0,t_i]$. The elements $u_{i+1}$ and $z_{i+1}$ are determined by the Iwasawa decomposition with respect to $G_{[t_i,t_i+\epsilon)}$. This determines the path $\phi$ for times infinitesimally after $t_i$; precisely, this determines the one-sided derivative $\phi_+'(t_i)$. Fixing $z_{i+1}$, we that the space of $\phi_+'(t_i)$ must lie in $ z_i^{-1} (U^- \cap G_{t_i}) z_{i + 1} G_{[t_i,t_i+\epsilon)} /G_{[t_i,t_i+\epsilon)}$.

Using the fact that $(U^- \cap G_{t_i}) = U^{nm-}_{t_i}$ \cite[Proposition 4.14 (P3)]{Rousseau-2016}, we see explicitly that $ z_i^{-1} (U^- \cap G_{t_i}) z_{i + 1} G_{[t_i,t_i+\epsilon)} /G_{[t_i,t_i+\epsilon)}$ can be identified with a finite-dimensional vector space corresponding to a finite set of affine roots. In particular, we have $ z_i^{-1} (U^- \cap G_{t_i}) z_{i + 1} G_{[t_i,t_i+\epsilon)} /G_{[t_i,t_i+\epsilon)} \cong  z_i^{-1} (U^- \cap P_{t_i}) z_{i + 1} P_{[t_i,t_i+\epsilon)} /P_{[t_i,t_i+\epsilon)}$. So the finite-dimensional vector space is identified with a finite-dimensional Schubert cell of the partial flag variety $P_{t_i}/P_{[t_i,t_i+\epsilon)}$. This partial flag variety is exactly the set of faces in in the tangent building at $t_i \lambda$ of type determined by $[t_i \lambda, (t_i+\eps)\lambda)$ (see \cite[\S 5.11 Proposition 4 and Remarks (b)]{Rousseau-2017}). This corresponds to the second paragraph of the proof of \cite[Theorem 6.3]{Gaussent-Rousseau-2008}. In particular, the points on the Schubert cell correspond to the set of minimal galleries considered there.

However, for $i>0$, we additionally need to impose the condition that $\phi$ is a straight line path in any apartment that contains it. That is, the condition that the one-sided derivative $\phi_-'(t_i)$, for times shortly before $t_i$,  agrees with $\phi_+'(t_i)$, the one-sided derivative shortly after $t_i$. This exactly corresponds to the big cell $P_{(t_i-\epsilon,t_i]} \cdot P_{[t_i,t_i+\epsilon)}/P_{[t_i,t_i+\epsilon)}$in the partial flag variety $P_{t_i}/P_{[t_i,t_i+\epsilon)}$ corresponding to the fundamental chamber in the building $\sI^-_{t_i}$. Note that this is because of our condition that $\lambda$ is dominant. In general, we will need to consider a Weyl translate of the big cell. The computation of the big-cell condition
corresponds to to the fifth paragraph of the proof of \cite[Theorem 6.3]{Gaussent-Rousseau-2008}.

We see that, at time $t_i$, the space of choices for $\phi_+'(t_i)$ is equal to a finite-dimensional Schubert cell in $P_{t_i}/P_{[t_i,t_i+\epsilon)}$ intersected with the big cell. Observe that this is a Deodhar set.

\subsubsection{Identifying the Deodhar sets}

Let us write:
\begin{align}
  \label{eq:59}
  z_{i-1}^{-1} U^- z_i \cdot P_{[t_i,t_i+\epsilon)} = z_{i-1}^{-1} U^- z_{i-1} z_{i-1}^{-1} z_i \cdot P_{[t_i,t_i+\epsilon)} 
\end{align}
We have:
\begin{align}
  \label{eq:60}
 z_{i-1}^{-1} U^- z_{i-1} = [z_{i-1}]^{-1} U^- [z_{i-1}]
\end{align}
By \eqref{eq:39-imported}, we have
\begin{align}
  \label{eq:61}
[z_{i-1}]^{-1} U^- [z_{i-1}] \cap P_{t_i} = \iota_{t_i}p_{t_i}([z_{i-1}])^{-1} U^- \iota_{t_i}p_{t_i}([z_{i-1}]) \cap P_{t_i}
=  \iota_{t_i}p_{t_i}([z_{i-1}])^{-1} B^-_{t_i} \iota_{t_i}p_{t_i}([z_{i-1}])
\end{align}
where $B^-_{t_i}$ is the fixator of the fundamental chamber in the tangent building $\sI^-_{t_i}$. In other words, $B^-_{t_i}$ is the parabolic in $P_{t_i}$ corresponding to the roots $\Phi^-_{t_i}$.

Secondly, the big cell in $ P_{t_i} /P_{[t_i,t_i+\epsilon)}  $ is equal to $B^+_{t_i} \cdot P_{[t_i,t_i+\epsilon)} /P_{[t_i,t_i+\epsilon)}  $ where $B^+_{t_i}$ is the parabolic corresponding to the roots $\Phi^+_{t_i}$. Therefore, the Deodhar set considered at time $t_i$ for $i>0$ in section \ref{sec:the-space-of-parameters} above is:
\begin{align}
  \label{eq:63}
 (\iota_{t_i}p_{t_i}([z_{i-1}])^{-1} B^-_{t_i} \iota_{t_i}p_{t_i}([z_{i-1}]) z_{i-1}^{-1}z_i \cdot   P_{[t_i,t_i+\epsilon)}  \cap B^+_{t_i} \cdot P_{[t_i,t_i+\epsilon)} )/P_{[t_i,t_i+\epsilon)} 
\end{align}
(Recall that for $i=0$ there is no big cell condition.)
This Deodhar set is expressed using the Weyl group $W_{t_i}$. Using the isomorphism $[\cdot]$, we can convert this to a Deodhar problem expressed using the the Weyl group $\Wv_{t_i}$, i.e.:
\begin{align}
  \label{eq:064}
 (p_{t_i}([z_{i-1}])^{-1} B^{v,-}_{t_i} p_{t_i}([z_{i-1}]) [z_{i-1}^{-1}z_i] \cdot   P^{v}_{t_i+\epsilon}  \cap B^{v,+}_{t_i} \cdot P^{v}_{t_i+\epsilon} )/P^{v}_{t_i+\epsilon} 
\end{align}
(Recall that $\iota_{t_i}$ is an inverse of $[\cdot]$, and that $p_{t_i}([z_{i-1}]), [z_{i-1}^{-1}z_i] \in \Wv_{t_i}$.)
The groups in \eqref{eq:064} are parabolic subgroups of $P_{t_i}$ corresponding to subsets of the roots dictated by the following table:
\begin{center}
\begin{tabular} {c c c}
  ``Parabolic'' group & Weyl group & Roots \\
   \hline  
  \hline
$P^v_{t_i}$  & $\Wv_{t_i}$ & $\Phi^v_{t_i}$ \\[0.5ex]
$P^v_{t_i+ \epsilon}$  & $\Wv_{t_i+\epsilon}$ & $\Phi^v_{t_i + \epsilon}$ \\[0.5ex]
$B^{v,-}_{t_i}$  & $\{e\}$ & $\Phi^{v,-}_{t_i}$ \\[0.5ex]
$B^{v,+}_{t_i}$  & $\{e\}$ & $\Phi^{v,+}_{t_i}$ \\[0.5ex]
   \hline
\end{tabular}     
\end{center}

By the work of \cite{Bardy}, the root system $\Phi^v_{t_i}$ is equal to the set of real roots in some Kac-Moody root system, and $\Phi^{v,+}_{t_i}$ is equal to the set of positive real roots in that root system. The Weyl group  $\Wv_{t_i}$ is the corresponding Weyl group; in particular it is a Coxeter group. Furthermore, $\Wv_{t_i+\epsilon}$ is a {\it standard} parabolic subgroup of $\Wv_{t_i}$, i.e., it is generated by a subset of the simple generators of $\Wv_{t_i}$.

We have explicitly described the Deodhar problem giving the cardinality of \eqref{eq:064}. In particular, we know that \eqref{eq:064} is non-empty if and only if 
\begin{align}
  \label{eq:66}
p_{t_i}([z_{i-1}]) <_{\Wv_{t_i}/\Wv_{t_i+\epsilon}} p_{t_i}([z_{i-1}]) [z_{i-1}^{-1}z_i]
\end{align}
where $\leq_{\Wv_{t_i}/\Wv_{t_i+\epsilon}}$ is the Bruhat order on the parabolic quotient 
$\Wv_{t_i}/\Wv_{t_i+\epsilon}$. Note that we have a strict equality above because of our procedure of choosing the folding times irredundantly. Furthermore, \eqref{eq:064} has dimension equal to:
\begin{align}
  \label{eq:65}
\ell_{\Wv_{t_i}/\Wv_{t_i+\epsilon}}\left(p_{t_i}([z_{i-1}]) [z_{i-1}^{-1}z_i]\right)
\end{align}
where $\ell_{\Wv_{t_i}/\Wv_{t_i+\epsilon}}$ is the length function for the parabolic quotient ${\Wv_{t_i}/\Wv_{t_i+\epsilon}}$.

\subsection{$U^-$-Hecke paths and the Kapovich-Milson folding condition}

Let us define the set of \emph{(spherical) $U^-$-Hecke paths of shape $\lambda$ and endpoint $\nu$} by:
\begin{align}
  \label{eq:8}
\cH^{\lambda,\nu}_{U^-,K} = \left\{ \rho_{U^-}(\phi) \suchthat \phi \in K.\phi_0 \textand \rho_{U^-}(\phi)(1) = \nu \right\}
\end{align}

Gaussent and Rousseau simply call such paths ``Hecke paths'', but we will emphasize $U^-$ as we will consider an $I_\infty$-version below. Gaussent and Rousseau (generalizing work of Kapovich and Millson \cite{Kapovich-Millson-2008}) describe a folding condition that characterize $U^-$-Hecke paths \cite[Definition 5.2]{Gaussent-Rousseau-2008}, which is a version of ``positive folding'' in the $p$-adic Kac-Moody setting. We will show below how to translate from the Kapovich-Millson folding condition to our condition \eqref{eq:66}. The reason we emphasize \eqref{eq:66} is that this condition will generalize most naturally to the case of $I_\infty$.

\subsubsection{Notational differences with Gaussent and Rousseau}

The main difference between our notations and those in \cite{Gaussent-Rousseau-2008} are the following:
\begin{itemize}
\item Our paths go from $0$ to $\lambda$, where those in \cite{Gaussent-Rousseau-2008} go from $-\lambda$ to $0$. So our time parameter is reversed.
\item Their folding condition involves the Weyl fixator $\Wv_{\tau(t_k)}$, whereas we work with $\Wv_{t_k \lambda}$. That is, they are looking at time $t_k$ on the retracted path, whereas we are looking at time $t_k$ on the straight-line path.
\item They pose their condition in the vectorial fixator  $\Wv_{\tau(t_k)}$. Working with the non-vectorial fixator is better for when we replace $U^-$ by $I_\infty$.
\item As mentioned before, we have $\pi^{\lambda}.0 = \lambda$, not $-\lambda$ as is the convention in \cite{Gaussent-Rousseau-2008}.
\end{itemize}

Aside from these notational differences, our notion of the retraction $\rho_{U^-}$ agree with the retraction of $\rho_{-\infty}$. Therefore, our notion of $U^-$-Hecke paths should agree. In particular, the Kapovich-Millson folding condition is equivalent to our condition \eqref{eq:66}. Let us understand the equivalence of the conditions combinatorially.

\subsubsection{The Kapovich-Millson folding condition in our notation}

Let us suppose that we are given the data of a prospective $U^-$-Hecke path, i.e. folding times $(0=t_0 < t_1 < \cdots t_N < t_{N+1} = 1)$ and folding directions $(z_0,\cdots,z_N)$. For each $k$, write $z_k = \pi^{\eta_k}w_k$.
The actual path $\tau : [0,1] \rightarrow \AA$ is given by
\begin{align}
  \label{eq:75}
  \tau(t) = z_k.t \lambda \text{ for } t \in [t_{k},t_{k+1}]
\end{align}
Let us fix $k$. For time shortly before $t_k$, we have
\begin{align}
  \label{eq:76}
  \tau(t) = z_{k-1}.t \lambda = t w_{k-1}(\lambda) + \eta_{k-1}
\end{align}
which implies that:
\begin{align}
  \label{eq:77}
  \tau'_-(t_k) =  w_{k-1}(\lambda)
\end{align}
Here $\tau'_-(t_k)$ is the one-sided derivative of $\tau$ at times before $t_k$. Similarly, we compute:
\begin{align}
  \label{eq:78}
  \tau'_+(t_k) =  w_k(\lambda)
\end{align}
In our notation, the Kapovich-Millson folding condition is \cite[Definition 5.2]{Gaussent-Rousseau-2008}
\begin{align}
  \label{eq:79}
\tau'_-(t_k) \leq_{\Wv_{\tau(t_k)}/\Wv_{\tau([t_k,t_k+\epsilon)}} \tau'_+(t_k)
\end{align}
which means there exists a sequence of positive roots $\beta_1, \cdots, \beta_s \in \Phi^{v,+}$, and elements $ \tau'_+(t_k) = \xi_{0} , \xi_{1}, \cdots, \xi_{s} = \tau'_-(t_k)$ satisfying:
\begin{itemize}
\item  $\xi_i = s_{\beta_i}(\xi_{i-1})$
\item  $ \langle \beta_i, \xi_{i-1} \rangle < 0$
\item  $\beta_i \in \Phi^{v}_{\tau(t_k)}$, i.e. $\langle \beta_i, \tau(t_k)  \rangle \in \ZZ$
\end{itemize}

The subscript ${\tau([t_k,t_k+\epsilon)}$ indicates that we are considering the Bruhat order in a parabolic quotient (this corresponds to the fact that we require the strict inequality $ \langle \beta_i, \xi_{i-1} \rangle < 0$ in the second condition above).

\subsubsection{Translating the folding condition}

Let us unpack the Kapovich-Millson folding condition. 

Equation \eqref{eq:79} means, using \eqref{eq:77} and \eqref{eq:78}, that there exists a sequence of positive roots $\beta_1, \cdots, \beta_s \in \Phi^{v,+}_{\tau(t_k)}$ and elements $w_k(\lambda) = \xi_0, \xi_1, \cdots, \xi_s= w_{k-1}(\lambda)$ satisfying:
\begin{itemize}
\item  $\xi_i = s_{\beta_i}(\xi_{i-1})$
\item  $ \langle \beta_i, \xi_{i-1} \rangle < 0$
\end{itemize}
Expanding the first condition, we have:
\begin{align}
  \label{eq:80}
  \xi_i = s_{\beta_i} s_{\beta_{i-1}} \cdots s_{\beta_1} w_k (\lambda)
\end{align}
Let us define $\zeta_1, \cdots, \zeta_s$ by:
\begin{align}
  \label{eq:81}
  \zeta_i = - w_k^{-1} s_{\beta_1} \cdots s_{\beta_{i-1}}(\beta_i)
\end{align}
The $ \langle \beta_i, \xi_{i-1} \rangle < 0$ condition above can be rewritten as:
\begin{align}
  \label{eq:82}
  \langle \zeta_i, \lambda \rangle > 0
\end{align}
We can also write:
\begin{align}
  \label{eq:84}
  \beta_i = - w_k s_{\zeta_1} \cdots s_{\zeta_{i-1}} (\zeta_i)
\end{align}
Let us rephrase the Kapovich-Millson folding condition in terms of $\zeta_1, \cdots, \zeta_s$.
The condition that $\beta_i \in \Phi^v_{\tau(t_k)}$, i.e.  $\langle \beta_i, \tau(t_k) \rangle \in \ZZ$, is equivalent to (we have $\tau (t_k) =   t w_k(\lambda) +\eta_k$, but the $\eta_k$ is irrelevant for the purpose of determining integrality):
\begin{align}
  \label{eq:83}
  \langle w_k s_{\zeta_1} \cdots s_{\zeta_{i-1}} (\zeta_i), t_k w_k(\lambda) \rangle  \in \ZZ
\end{align}
Therefore, $s_{\zeta_1} \cdots s_{\zeta_{i-1}} (\zeta_i) \in \Phi^{v}_{t_k}$ for all $i$. By induction, we have $\zeta_i \in \Phi^{v}_{t_k}$ for all $i$. Condition \eqref{eq:82} implies (because $\lambda$ is dominant) that $\zeta_i \in \Phi^{v,+}_{t_k}$ and that $s_{\zeta_i}$ is a non-trivial element of the quotient $\Wv_{t_k}/\Wv_{t_k+\epsilon}$.

Therefore, we can rewrite the condition in the following way. There exists a sequence of positive roots $\zeta_1, \cdots, \zeta_s \in \Phi^{v,+}_{t_k}$ such that $\langle \zeta_i, \lambda \rangle \neq 0$ and a chain in the Bruhat order of $\Wv$:
\begin{align}
  \label{eq:85}
  w_k \leftarrow w_k s_{\zeta_1} \leftarrow w_k s_{\zeta_1} s_{\zeta_2} \leftarrow \cdots \leftarrow w_k s_{\zeta_1} s_{\zeta_2} \cdots s_{\zeta_s} = w_{k-1}
\end{align}

Explicitly, the condition that $\beta_1, \cdots, \beta_s \in \Phi^{v}_{\tau(t_k)}$ is equivalent to the condition that $\zeta_1, \cdots, \zeta_s \in \Phi^{v}_{t_k}$. The condition that $\langle \beta_i, \xi_{i-1} \rangle > 0$ is equivalent to the condition that $\zeta_1, \cdots, \zeta_s \in \Phi^{v,+}$ and $\langle \zeta_i, \lambda \rangle \neq 0$. Finally, the condition that $\beta_1, \cdots, \beta_s > 0$ is equivalent to the condition that \eqref{eq:85} is a chain in the Bruhat order.

\subsubsection{Comparing with our condition}

Let us define $\overline{w_{k-1}} \in \Wv_{t_k}$ by:
\begin{align}
  \label{eq:87}
 \overline{w_{k-1}}^{-1} ( \Phi^{v,+}_{t_k})  =  w_{k-1}^{-1}( \Phi^{v,+}) \cap \Phi^v_{t_k} 
\end{align}
Then our condition \eqref{eq:66} is:
\begin{align}
  \label{eq:88}
 \overline{w_{k-1}} \leq_{\Wv_{t_k}/\Wv_{t_k+\epsilon}} \overline{w_{k-1}} w_{k-1}^{-1} w_k 
\end{align}
Let us rewrite \eqref{eq:85} as:
\begin{align}
  \label{eq:86}
  w_{k-1} \rightarrow w_{k-1}s_{\zeta_s} \rightarrow w_{k-1}s_{\zeta_s}s_{\zeta_{s-1}} \rightarrow \cdots \rightarrow w_{k-1}s_{\zeta_s}s_{\zeta_{s-1}}\cdots{s_{\zeta_1}} = w_k
\end{align}
For all $i$, $w_{k-1} s_{\zeta_s} \cdots s_{\zeta_{i+1}}(\zeta_i) \in \Phi^{v,+}$ by \eqref{eq:86}. Moreover,
\begin{align}
  \label{eq:89}
 w_{k-1}^{-1} w_{k-1} s_{\zeta_s} \cdots s_{\zeta_{i+1}}(\zeta_i) = s_{\zeta_s} \cdots s_{\zeta_{i+1}}(\zeta_i) \in \Phi^{v}_{t_k}
\end{align}
Therefore, $s_{\zeta_s} \cdots s_{\zeta_{i+1}}(\zeta_i)$ is an element of the right-hand side of \eqref{eq:87}. Therefore, there exists $\gamma_i \in \Phi^{v,+}_{t_k}$ such that:
\begin{align}
  \label{eq:90}
 s_{\zeta_s} \cdots s_{\zeta_{i+1}}(\zeta_i)  = \overline{w_{k-1}}^{-1}(\gamma_i)
\end{align}
Equivalently, for all $i$:
\begin{align}
  \label{eq:91}
 \overline{w_{k-1}} s_{\zeta_s} \cdots s_{\zeta_{i+1}}(\zeta_i) \in \Phi^{v,+}_{t_k}
\end{align}
Therefore \eqref{eq:86} is equivalent to:
\begin{align}
  \label{eq:92}
  \overline{w_{k-1}} \rightarrow \overline{w_{k-1}} s_{\zeta_s} \rightarrow \overline{w_{k-1}}s_{\zeta_s}s_{\zeta_{s-1}} \rightarrow \cdots \rightarrow \overline{w_{k-1}}s_{\zeta_s}s_{\zeta_{s-1}}\cdots{s_{\zeta_1}} = \overline{w_{k-1}} w_{k-1}^{-1}w_k
\end{align}
This is a chain in the Bruhat order of $\Wv_{t_k}$. Because of our condition that $\langle \zeta_i, \lambda \rangle \neq 0$, this descends to a chain in the Bruhat order of the quotient $\Wv_{t_k}/\Wv_{t_k+\epsilon}$. Therefore we see that our condition \eqref{eq:88} is equivalent to the Kapovich-Millson folding condition.

\subsubsection{Finiteness of the set of $U^-$-Hecke paths}

As we have discussed above, the elements $\tau \in \cH^{\lambda,\nu}_{U^-}$ are determined by the data of a finite sequence of folding times $0=t_0 < t_1 < \cdots < t_N < t_{N+1} = 1$ and a finite sequence of folding directions $z_0,\ldots,z_{N} \in W = W_{Q}$ subject to the inequalities \eqref{eq:66}. If we ask that the inequalities in \eqref{eq:66} are strict, this data is uniquely determined. 

Fix $\tau \in \cH^{\lambda,\nu}_{U^-}$. We have seen that the set of $\phi \in K.\phi_0$ such that $\rho_{U^-}(\phi) = \tau$ is given by an explicit a product of $N$ Deodhar sets determined by the folding times and directions. In particular, when $\kk$ is a finite field, this set is finite with cardinality varying polynomially in $q = \#\kk$.

Furthermore, Gaussent and Rousseau prove \cite[Corollary 5.9]{Gaussent-Rousseau-2008} that the set $\cH^{\lambda,\nu}_{U^-,K}$ is \emph{finite} using a generalization of the combinatorics controlling Lakshmibai-Seshadri paths. Therefore, they can conclude that the set is finite, and hence they obtain an explicit decomposition of $(K \pi^\lambda K \cap U^- \pi^\nu K)/K$ into a finite number of disjoint subsets (indexed by $\cH^{\lambda,\nu}_{U^-,K}$) each of whose cardinalities is given by product of explicit Deodhar problems. 

Our above discussion shows that the set $\cH^{\lambda,\nu}_{U^-,K}$ is controlled by certain chains in the Bruhat order. It seems natural to have a general proof of finiteness of $\cH^{\lambda,\nu}_{U^-,K}$ using only properties of the Bruhat order. We have not pursued this here.

However, below when we replace $U^-$ by $I_\infty$, this will be our approach to showing finiteness of the analogous notion (but we will work with the far more mysterious $G$-affine Bruhat order). We therefore hope that there should be a uniform method using variants of $G$-affine Bruhat orders to show finiteness of sets of $U^-$-Hecke paths that covers the $U^-$ and $I_\infty$ (and more general) cases.

\section{Computing Spherical $R$-polynomials}
\label{sec:computing-spherical-R-polynomials}

Let $\lambda$ and $\nu$ be dominant weights in $\overset{\circ}{\cT}$. Our goal is to compute the cardinality of
\begin{align}
  \label{eq:131}
(K \pi^\lambda K \cap I_\infty \pi^\nu K )/K
\end{align}
However, because we have not yet constructed the retraction $\rho_{I_\infty}$, we cannot yet accomplish this. Nonetheless, using the case of $\rho_{U^-}$ and the Satake transform as a guide, we can compute the expected answer. 

Below, we will proceed assuming that $\rho_{I_\infty}$ is defined and satisfies formal properties analogous to those of $\rho_{U^-}$. Our computation can therefore be taken as a definition of $R$-polynomials. The main mathematical result is that our definition is well-defined (i.e. it involves a finite sum).

\subsection{Computing the retraction $\rho_{I_\infty}$}

Let $\phi_0$ be the straight-line path from $0$ to $\lambda$, and as before we identify $K . \phi_0$ with $(K \pi^\lambda K )/K$. Let $\phi \in K . \phi_0$. Then there exists $g \in K$ such that:  
\begin{align}
  \label{eq:013}
  \phi(t) = (g, t \pi^\lambda.0)
\end{align}
Analogously to the computation in \S \ref{sec:computing-retractions-rho-infty}, we produce folding times $0 = t_0 < t_1 < \cdots t_N < t_{N+1} = 1$ such that
\begin{align}
  \label{eq:018}
  g = i_0 \cdots i_k x_k p_k
\end{align}
where $i_0, \cdots, i_k \in I_\infty$ and $p_k \in G_{[t_k,t_{k+1}]}$.

\subsubsection{The retracted path}
Let $\tau = \rho_{I_\infty}(\phi)$ be the retracted path. We have:
\begin{align}
  \label{eq:020}
 \tau(t) = x_k.( t \lambda) \text{ for } t \in [t_k,t_{k+1}]
\end{align}
Finally, the condition that $\phi$ corresponds to a point of $I_\infty \pi^\nu K /K$ becomes:
\begin{align}
  \label{eq:021}
 x_N . \lambda = \nu 
\end{align}

\subsection{The Deodhar sets}

Let us fix a retracted path $\tau$ as above in \eqref{eq:020}. This corresponds to a sequence $0 = t_0 < t_1 < \cdots t_N < t_{N+1} = 1$ of folding times and  folding directions $x_0, \cdots x_N$. We then consider the set of $\phi \in K.\phi_0$ such that $\rho_{I_\infty}(\phi) = \tau$ where $\tau$ is defined by \eqref{eq:020}.

The set of such $\phi$ are determined by the $i_k$ that appear in \eqref{eq:018}. Because the $i_k$ arise from Schubert cells, we can arrange it so that the $i_k$ are uniquely determined by requiring that they lie in a specified subset of $I_\infty$ determined by $x_{k-1}$ and $x_k$.
The set of possible choices of $i_0$ is the Schubert cell:
\begin{align}
  \label{eq:21}
   I_\infty x_0 \cdot P_{[0,0+ \epsilon)} /  P_{[0,0+ \epsilon)} = I_\infty x_0 \cdot P_{[t_0,t_0+ \epsilon)} / P_{[t_0,t_0+ \epsilon)} \subseteq P_{t_0}/P_{[t_0,t_0+ \epsilon)}
\end{align}
For $k >0$, the set of possible choices of $i_k$ is the following set
\begin{align}
  \label{eq:19}
 (x_{k-1}^{-1} I_\infty x_k \cdot P_{[t_k,t_k+\epsilon)} \cap  P_{(t_{k}-\epsilon,t_k]} \cdot P_{[t_k,t_k+\epsilon)} )/P_{[t_k,t_k+\epsilon)} \subseteq P_{t_k}/P_{[t_k,t_k+\epsilon)}
\end{align}
which is the intersection of a Schubert cell with the big cell. 
As in the $U^-$ case, the big cell condition arises because $\phi$ is a straight path in the masure, i.e. its two sided derivatives at $t_k$ must agree.

\subsubsection{A finiteness result}

For the following proof we will use of our assumption that $\bG$ is affine and that $\overset{\circ}{\cT}$ consists of positive level coweights.  However, we expect that this restriction is unnecessary.

\begin{Proposition}
  \label{proposition:existence-of-overline-x-k}
  There exists unique $\overline{x_{k-1}} \in W_{t_k}$ such that:
  \begin{align}
    \label{eq:22}
     x_{k-1}^{-1}(\Phi^-) \cap \Phi_{t_k} = \overline{x_{k-1}}^{-1}(\Phi_{t_k}^{-})
  \end{align}
  
\end{Proposition}

\begin{proof}
  We will prove that $x_{k-1}^{-1}(\Phi^-) \cap \Phi^{+}_{t_k}$ is finite. Furthermore, it is clearly a biconvex subset of $\Phi^{+}_{t_k}$. (Recall that a subset $S \subset \Phi^{+}_{t_k}$ is convex if for all $\alpha, \beta \in S$, we have $\alpha + \beta \in S$. The set $S$ is biconvex if it is convex and 
$\Phi^{U^+}_{t_k} \backslash S$ is also convex.)

By a theorem of Papi \cite{Papi-1994}, it follows that $x_{k-1}^{-1}(\Phi^-) \cap \Phi^{+}_{t_k} = \overline{x_{k-1}}^{-1}(\Phi_{t_k}^{-}) \cap \Phi_{t_k}^{+}$ for a unique $\overline{x_{k-1}} \in W_{t_k}$. Because $x_{k-1}^{-1}(\Phi^-) \cap \Phi^{}_{t_k}$ includes either every root in $\Phi_{t_k}$ or its negative, but not both, we also have $x_{k-1}^{-1}(\Phi^-) \cap \Phi^{-}_{t_k} = \overline{x_{k-1}}^{-1}(\Phi_{t_k}^{-}) \cap \Phi_{t_k}^{+}$. Therefore we can conclude that $ x_{k-1}^{-1}(\Phi^-) \cap \Phi_{t_k} = \overline{x_{k-1}}^{-1}(\Phi_{t_k}^{-})$.

  We therefore have the task of proving that $x_{k-1}^{-1}(\Phi^-) \cap \Phi^{+}_{t_k}$ is finite. Write $x_{k-1} = w \pi^\eta$ where $w \in \Wv$ and $\eta$ has level zero. Then equivalently we will show that $w\pi^\eta (\Phi^{+}_{t_k}) \cap \Phi^-$ is finite. We have
  \begin{align}
    \label{eq:13}
  \Phi^{+}_{t_k} = \left \{ \beta[t_k \langle \lambda, \beta \rangle ] \suchthat t_k \langle \lambda, \beta \rangle \in \ZZ \textand \beta \text{ is a positive real root }\right \}
  \end{align}
  Because $\lambda$ is dominant, we have $\beta[t_k \langle \lambda, \beta \rangle ] = \beta + t_k \langle \lambda, \beta \rangle \pi$. Even if $\lambda$ were not dominant, but merely lived in the Tits cone, this condition is true for almost all $\beta$.

  Compute
  \begin{align}
    \label{eq:14}
  w\pi^\eta (\beta + t_k \langle \lambda, \beta \rangle \pi) =  w(\beta) + (\langle \eta, \beta \rangle + t_k \langle \lambda, \beta \rangle )\pi = w(\beta) + (\langle \eta + t_k  \lambda, \beta \rangle )\pi
  \end{align}

  As $\lambda$ has positive level, $\eta + t_k  \lambda$ has the same positive level and lies in the Tits cone. Therefore, $\langle \eta + t_k  \lambda, \beta \rangle$ is greater than or equal to zero for almost all positive real roots $\beta$. Finally, observe that $w(\beta)$ is a positive real root for almost all $\beta$.
\end{proof}

\subsubsection{Identifying the Deodhar problem}

Let $B^-_{t_k}$ be the fixator of the local chamber $(t_k-\epsilon C_f, t_k]$, where $(t_k-\epsilon C_f, t_k]$ denotes the chamber germ based at $t_k \lambda$ in the direction of the anti-dominant vectorial chamber $-C_f$.

The group theoretic analogue of Proposition \ref{proposition:existence-of-overline-x-k}, is the following statement, which we will take as an assumption toward our definition of $R$-polynomials. A proof of this statement will required a detailed analysis of the group $I_\infty$, which we will take up in a future work.
\begin{Assumption}
  \begin{align}
    \label{eq:59}
x_{k-1}^{-1} I_\infty  x_{k-1} \cap P_{t_k} = \overline{x_{k-1}}^{-1} B^-_{t_k} \overline{x_{k-1}}
  \end{align}
\end{Assumption}
Observe that a special case of this assumption is that $B^-_{t_k} = I_\infty \cap P_{t_k}$. With this assumption we see that \eqref{eq:19} is a Deodhar problem. The relevant Weyl groups and roots are given in the table below.
\begin{center}
\begin{tabular} {c c c}
    Parahoric group & Weyl group & Roots \\
   \hline  
  \hline
$P_{t_k}$  & $W_{t_k}$ & $\Phi_{t_k}$ \\[0.5ex]
$P_{[t_k, t_k + \epsilon)}$  & $W_{[t_k,x_k+\epsilon)}$ & $\Phi_{[t_k, t_k + \epsilon)}$ \\[0.5ex]
$P_{(t_k-\epsilon, t_k] }$  & $W_{(t_k-\epsilon, t_k]}$ & $\Phi_{(t_k-\epsilon, t_k]}$ \\[0.5ex]
$B^-_{t_k}$  & $\{e\}$ & $\Phi_{(t_k-\epsilon C_f, t_k]}$ \\[0.5ex]
   \hline
\end{tabular}     
\end{center}
We can then rewrite \eqref{eq:19} as:
\begin{align}
  \label{eq:23}
 (\overline{x_{k-1}}^{-1} B^-_{t_k} \overline{x_{k-1}} x_{k-1}^{-1}x_k \cdot P_{[t_k,t_k+\epsilon)} \cap  P_{(t_{k}-\epsilon,t_k]} \cdot P_{[t_k,t_k+\epsilon)} )/P_{[t_k,t_k+\epsilon)} \subseteq P_{t_k}/P_{[t_k,t_k+\epsilon)}
\end{align}
This set is non-empty if and only if
\begin{align}
  \label{eq:24}
\overline{x_{k-1}} \leq_{W_{t_k}/W_{[t_k,t_k+\epsilon)}} \overline{x_{k-1}} x_{k-1}^{-1}x_k  
\end{align}
and it has dimension equal to:
\begin{align}
  \label{eq:54}
\ell_{W_{t_k}/W_{[t_k,t_k + \epsilon)} }(  \overline{x_{k-1}} x_{k-1}^{-1}x_k)  
\end{align}
Moreover, we can explicitly count the size of \eqref{eq:23} using distinguished subexpressions.

The condition that $\lambda$ is dominant is not crucial to the discussion above. In \S \ref{sec:iwahori-spherical-r-polynomials} below, we will explain how to relax the dominance condition.

\subsection{The folding condition for $I_\infty$-Hecke paths}

\subsubsection{$I_\infty$-Hecke paths}
Let us define the set of {\it (spherical) $I_\infty$-Hecke paths} of shape $\lambda$ and endpoint $\nu$ by:
\begin{align}
  \label{eq:25}
\cH^{\lambda,\nu}_{I_\infty} = \{ \rho_{I_\infty}(\phi) \suchthat \phi \in K.[0,\lambda] \textand \rho_{I_\infty}(\phi)(1)=\nu \}
\end{align}
Combinatorially, $\tau \in \cH^{\lambda,\nu}_{I_\infty}$ is given by a set of times $0 = t_0 < t_1 < \cdots t_N < t_{N+1} = 1$ and elements $x_k \in W$ such that 
\begin{align}
  \label{eq:26}
x_{k-1}^{-1}x_k \in W_{t_k}  
\end{align}
and a {\it folding condition} 
\begin{align}
  \label{eq:24}
\overline{x_{k-1}} \leq_{W_{t_k}/W_{[t_k,t_k+\epsilon)}} \overline{x_{k-1}} x_{k-1}^{-1}x_k  
\end{align}
and:
\begin{align}
  \label{eq:27}
 x_N. \lambda = \nu 
\end{align}
To not double count, we further require
\begin{align}
  \label{eq:28}
  x_{k-1}. \lambda \neq x_k . \lambda
\end{align}
for all $k$.

\subsubsection{Minimal length representatives}

Although only the class of $x_{k-1}$ in $W/W_{[t_k, t_k+ \epsilon)}$ matters for the purpose of specifying an {\it $I_\infty$-Hecke path}, it will be useful for our computation below to choose a specific representative $\tilde{x_{k-1}} \in W$ that plays the role of a ``minimal length representative''. Let $\tilde{\overline{x_{k-1}}}$ be the minimal length representative of $\overline{x_{k-1}}$ in $W_{t_k}$ modulo $W_{[t_k,t_k+\epsilon)}$. Then we define $\tilde{x_{k-1}} = x_{k-1} \overline{x_{k-1}}^{-1} \tilde{\overline{x_{k-1}}}$.
This definition guarantees that:
\begin{align}
  \label{eq:115}
     \tilde{x_{k-1}}^{-1}(\Phi^-) \cap \Phi_{t_k} = \tilde{\overline{x_{k-1}}}^{-1}(\Phi_{t_k}^{-})
\end{align}
From now on we will drop the tildes and assume that $x_{k-1}$ was initially chosen so that $\tilde{\overline{x_{k-1}}} = \overline{x_{k-1}}$.

Observe that the minimal length condition forces $\overline{x_0} = e$. It is conceivable that the minimal length condition also forces $\overline{x_k} = e$, but we do not currently see a proof.

\subsubsection{The folding condition}
Let us fix an element $\tau \in \cH^{\lambda,\nu}_{I_\infty}$ as above, and let us keep the $t_k$ and $x_k$ notation as above. We will try to reinterpret the folding condition $\overline{x_{k-1}} \leq_{W_{t_k}/W_{[t_k,t_k+\epsilon)}} \overline{x_{k-1}} x_{k-1}^{-1}x_k$ in a more useful form. Because of our minimal length assumption for $\overline{x_{k-1}}$, we can find a sequence $\zeta_1[r_1], \cdots, \zeta_s[r_s] \in \Phi^{+}_{t_k} \backslash \Phi^{+}_{[t_k,t_k+\epsilon)}$ (equivalently $\langle \zeta_i, \lambda \rangle \neq 0$) such that:
\begin{align}
  \label{eq:29}
  \overline{x_{k-1}} \underset{\Phi^+}{\rightarrow}  \overline{x_{k-1}}s_{\zeta_1[r_1]} \underset{\Phi^+}{\rightarrow} \overline{x_{k-1}}s_{\zeta_1[r_1]}s_{\zeta_2[r_2]}  \underset{\Phi^+}{\rightarrow} \cdots
\underset{\Phi^+}{\rightarrow} \overline{x_{k-1}}s_{\zeta_1[r_1]}s_{\zeta_2[r_2]}\cdots{s_{\zeta_s[r_s]}} = 
\overline{x_{k-1}} x_{k-1}^{-1}x_k 
\end{align}
Here
\begin{align}
  \label{eq:31}
\overline{x_{k-1}}s_{\zeta_1[r_1]}s_{\zeta_2[r_2]}\cdots{s_{\zeta_{j-1}[r_{j-1}]}} 
\underset{\Phi^+}{\rightarrow} \overline{x_{k-1}}s_{\zeta_1[r_1]}s_{\zeta_2[r_2]}\cdots{s_{\zeta_j[r_j]}}
\end{align}
means:
\begin{align}
  \label{eq:30}
 \overline{x_{k-1}}s_{\zeta_1[r_1]}s_{\zeta_2[r_2]}\cdots{s_{\zeta_{j-1}[r_{j-1}]}}(\zeta_{j}[r_{j}]) \in \Phi^+
\end{align}
Observe that we in fact have
\begin{align}
  \label{eq:33}
 \overline{x_{k-1}}s_{\zeta_1[r_1]}s_{\zeta_2[r_2]}\cdots{s_{\zeta_{j-1}[r_{j-1}]}}(\zeta_{j}[r_{j}]) \in \Phi^+_{t_k}
\end{align}
because $\overline{x_{k-1}} \in W_{t_k}$.
Recall that $\overline{x_{k-1}} \in W_{t_k}$ satisfies the condition:
\begin{align}
  \label{eq:32}
 \overline{x_{k-1}}^{-1}(\Phi^+_{t_k}) = x_{k-1}^{-1}(\Phi^+) \cap  \Phi_{t_k} 
\end{align}
Therefore  \eqref{eq:33} is equivalent to
\begin{align}
  \label{eq:34}
 s_{\zeta_1[r_1]}s_{\zeta_2[r_2]}\cdots{s_{\zeta_{j-1}[r_{j-1}]}}(\zeta_{j-1}[r_{j-1}]) \in  x_{k-1}^{-1}(\Phi^+) \cap  \Phi_{t_k}
\end{align}
which is equivalent to:
\begin{align}
  \label{eq:35}
 x_{k-1}s_{\zeta_1[r_1]}s_{\zeta_2[r_2]}\cdots{s_{\zeta_{j-1}[r_{j-1}]}}(\zeta_{j-1}[r_{j-1}]) \in \Phi^+
\end{align}
Therefore \eqref{eq:29} is equivalent to:
\begin{align}
  \label{eq:36}
  x_{k-1} \underset{\Phi^+}{\rightarrow}  x_{k-1}s_{\zeta_1[r_1]} \underset{\Phi^+}{\rightarrow} x_{k-1}s_{\zeta_1[r_1]}s_{\zeta_2[r_2]}  \underset{\Phi^+}{\rightarrow} \cdots
\underset{\Phi^+}{\rightarrow} x_{k-1}s_{\zeta_1[r_1]}s_{\zeta_2[r_2]}\cdots{s_{\zeta_s[r_s]}} = 
x_{k-1} x_{k-1}^{-1}x_k=x_k
\end{align}
Let us write
\begin{align}
  \label{eq:37}
  x_{k-1} \underset{W_{t_k}/W_{[t_k,t_k+\epsilon)},\Phi^+}{\leq} x_k
\end{align}
for this condition. Remember that implicitly we are choosing $x_{k-1}$ to satisfy the minimal length conditions above.

\subsubsection{Relationship with the Bruhat order on $\WTits$}

\begin{Proposition}
  \label{prop:folding-condition-equiv-to-bruhat-order-condition}
  Let $t \in [0,1)$, $x \in W$. Suppose $\zeta[r] \in \Phi^+_{t} \backslash \Phi^+_{[t,t+\epsilon)}$ (i.e. $\langle \zeta, \lambda \rangle \neq 0$. Then
  \begin{align}
    \label{eq:38}
   x  \underset{\Phi^+}{\rightarrow} x s_{\zeta[r]}
  \end{align}
  if and only if:
  \begin{align}
    \label{eq:39}
   x \pi^\lambda   \underset{\Phi^+}{\leftarrow} x s_{\zeta[r]} \pi^\lambda
  \end{align}
\end{Proposition}
\begin{proof}
  First $\zeta[r] \in \Phi^+_{t}$ occurs if and only if $r = t \langle \zeta, \lambda \rangle$ and $t \langle \zeta, \lambda \rangle \in \ZZ$. Because $\lambda$ is dominant and $\langle \zeta, \lambda \rangle \neq 0$, we have $\langle \zeta,\lambda \rangle > 0$.
  Recall  $x  \underset{\Phi^+}{\rightarrow} x s_{\zeta[r]}$ means:
  \begin{align}
    \label{eq:43}
    x(\zeta[r]) \in \Phi^+
  \end{align}
  Next, let us write:
  \begin{align}
    \label{eq:44}
  x s_{\zeta[r]} \pi^\lambda = x \pi^\lambda \pi^{-\lambda} s_{\zeta[r]} \pi^\lambda = x \pi^\lambda s_{ \pi^{-\lambda}(\zeta[r]) }
  \end{align}
  We compute, using the fact that $\lambda$ is dominant and that $t < 1$:
  \begin{align}
    \label{eq:45}
    \pi^{-\lambda}(\zeta[r]) = \pi^{-\lambda}(\zeta + t\langle\zeta,\lambda\rangle\pi) = \zeta + (t-1)\langle\zeta,\lambda\rangle\pi \in \Phi^-
  \end{align}
  Therefore \eqref{eq:39} holds if and only if (take careful note of the signs):
  \begin{align}
    \label{eq:46}
x \pi^\lambda (\zeta + (t-1)\langle\zeta,\lambda\rangle\pi) \in \Phi^+
  \end{align}
  We compute:
  \begin{align}
    \label{eq:47}
x \pi^\lambda (\zeta + (t-1)\langle\zeta,\lambda\rangle\pi) =    
x ( \zeta + t \langle\zeta,\lambda\rangle\pi) = x(\zeta[r])
  \end{align}
 Therefore, \eqref{eq:39} holds if and only if \eqref{eq:43} holds, which we have seen is equivalent to \eqref{eq:38}. 
\end{proof}

\begin{Corollary}
  \label{cor:equiv-of-two-ineq}
  For all $k$, we have
  \begin{align}
    \label{eq:40}
  x_{k-1} \underset{W_{t_k}/W_{[t_k,t_k+\eps)},\Phi^+}{\leq} x_k
  \end{align}
  if and only if:
  \begin{align}
    \label{eq:41}
  x_{k-1}\pi^\lambda \underset{W_{t_k-1}/W_{[t_k-1,t_k-1+\epsilon)},\Phi^+}{\geq} x_k\pi^\lambda
  \end{align}
\end{Corollary}
\noindent Note that \eqref{eq:41} implies a inequality in the Bruhat order on $\WTits$. 

\subsection{Finiteness of the set of $I_\infty$-Hecke paths}
By Corollary \ref{cor:equiv-of-two-ineq}, the data of our Hecke path satisfies:
\begin{align}
  \label{eq:42}
  \pi^\lambda \underset{W_{t_0-1}/W_{[t_0-1,t_0-1+\epsilon)},\Phi^+}{\geq} x_0\pi^\lambda \underset{W_{t_1-1}/W_{[t_1-1,t_1-1+\epsilon},\Phi^+}{\geq} \cdots x_{N-1}\pi^\lambda \underset{W_{t_N-1}/W_{[t_N-1,t_N-1+\epsilon)},\Phi^+}{\geq} x_N\pi^\lambda
\end{align}
The non-redundancy condition \eqref{eq:28} tells us that all but the first inequality is strict:
\begin{align}
  \label{eq:48}
  \pi^\lambda \underset{W_{t_0-1}/W_{[t_0-1,t_0-1+\epsilon)},\Phi^+}{\geq} x_0\pi^\lambda \underset{W_{t_1-1}/W_{[t_1-1,t_1-1+\epsilon},\Phi^+}{>} \cdots x_{N-1}\pi^\lambda \underset{W_{t_N-1}/W_{[t_N-1,t_N-1+\epsilon)},\Phi^+}{>} x_N\pi^\lambda
\end{align}

The endpoint condition \eqref{eq:27} means that
\begin{align}
  \label{eq:49}
 x_N \pi^\lambda = \pi^\nu w 
\end{align}
for some $w \in \Wv$.
Therefore:
\begin{align}
  \label{eq:50}
  \pi^\lambda \geq \pi^\nu w 
\end{align}

\subsubsection{Finiteness of the set of $x_N$}
By \cite[Theorem 3.3]{Muthiah-Orr-2018}, this implies
\begin{align}
  \label{eq:51}
  \ell(\pi^\lambda) \geq \ell(\pi^\nu w)
\end{align}
where $\ell$ is the the length function from \cite{Muthiah-2018} specialized with $\epsilon=1$.

Since $\lambda$ is dominant, $\ell(\pi^\lambda) = \langle 2\rho, \lambda \rangle$, where $\rho$ is the sum of fundamental weights. As we have assumed that $\nu$ is dominant, we have:
\begin{align}
  \label{eq:52}
\ell(\pi^\nu w) = \langle 2\rho, \nu \rangle  + \ell(w)
\end{align}
Therefore, we must have
\begin{align}
  \label{eq:053}
  \ell(w) \leq \langle 2 \rho, \lambda - \nu \rangle
\end{align}
which shows that the set of possible choices for $w$, and hence for $x_N$ is finite.

\begin{Remark}
  The bound \eqref{eq:053} is reminiscent of an estimate due to Joseph. The Joseph estimate plays a key role in showing that spaces of rational maps into Kac-Moody flag varieties of fixed degree are finite-type schemes \cite{Braverman-Finkelberg-Gaitsgory-2006}. It also a key finiteness result on the way to proving the affine Gindikin-Karpelevich formula \cite{Braverman-Garland-Kazhdan-Patnaik}. A similar estimate appears in the proof of the finiteness of $U^-$-Hecke paths \cite{Gaussent-Rousseau-2008}.
\end{Remark}

\begin{Remark}
If we drop the assumption that $\nu$ is dominant, then \eqref{eq:52} shifts by a fixed amount related to the length of the element of $\Wv$ that sends $\nu$ to a dominant weight. In particular, we will still obtain a bound on $\ell(w)$, from which we can make the same finiteness conclusions as above.
\end{Remark}

\subsubsection{Finiteness of $\cH^{\lambda,\nu}_{I_\infty}$}
We have seen that ever element $\tau \in \cH^{\lambda,\nu}_{I_\infty}$ gives rise to a sequence
\begin{align}
  \label{eq:56}
  \pi^\lambda \underset{W_{t_0-1}/W_{[t_0-1,t_0-1+\epsilon)},\Phi^+}{\geq} x_0\pi^\lambda \underset{W_{t_1-1}/W_{[t_1-1,t_1-1+\epsilon},\Phi^+}{>} \cdots x_{N-1}\pi^\lambda \underset{W_{t_N-1}/W_{[t_N-1,t_N-1+\epsilon)},\Phi^+}{>} x_N\pi^\lambda
\end{align}
in the Bruhat order on $\WTits$. By \eqref{eq:26}, the folding times are determined by the set of $x_k$, so the sequence \eqref{eq:56} determines the path $\tau$. Furthermore, we have proved that the set of possible $x_N$ is finite. 

\newcommand{\ASL}{\widehat{\SL}}

In general, Orr and the author have proved that the length of intervals in the double affine Bruhat order are controlled by the length function \cite[Theorem 3.3]{Muthiah-Orr-2018}. Moreover, in the case of untwisted affine ADE we show that covers are characterized by the length functions \cite[Theorem 6.2]{Muthiah-Orr-2018}. Building on that characterization of covers, Welch has given a beautiful explicit classification of covers in the double affine Bruhat order (for untwisted affine ADE) \cite[Theorem 5.4.8]{Welch-2019}. In particular, she shows that the set of covers of a given element of $W_\cT$ is finite \cite[Corollary 5.4.6]{Welch-2019}. Summarizing we have the following.

\begin{Theorem}[\cite{Muthiah-Orr-2018},\cite{Welch-2019}]
  In untwisted affine ADE type, intervals in the double affine Bruhat order are finite. Precisely, given $x,y \in W_{\cT}$ with $x \leq y$, the set of $z \in W_{\cT}$ such that $x \leq z \leq y$ is a finite set.
\end{Theorem}

Combining this with our calculation above that the set of possible $x_N$ is finite, we therefore deduce our main finiteness result.
\begin{Theorem}
  \label{thm:finiteness-of-I-infty-Hecke-paths}
  In untwisted affine ADE type, the set $\cH^{\lambda,\nu}_{I_\infty}$ is finite.
\end{Theorem}

\subsubsection{The definition of spherical $R$-polynomials}

We now can summarize our calculation as a definition of $R$-polynomials.
\begin{Definition}
  \label{def:spherical-R-poly}
Let $\lambda$ and $\nu$ be dominant coweights. Then we define the \emph{spherical $R$-polynomial} by: 
\begin{align}
  \label{eq:58}
  R^K_{\nu,\lambda} = \sum_{\tau \in \cH^{\lambda,\nu}_{I_\infty}} | I_\infty x_0 \cdot P_{[t_0,t_0+ \epsilon)}/P_{[t_0,t_0+ \epsilon)}| \cdot \prod_{i=1}^N
  |x_{k-1}^{-1} I_\infty x_k \cdot P_{[t_k,t_k+\epsilon)} \cap  P_{(t_{k}-\epsilon,t_k]} \cdot P_{[t_k,t_k+\epsilon)}/ P_{(t_{k}-\epsilon,t_k]}|
\end{align}
\end{Definition}
By Theorem \ref{thm:finiteness-of-I-infty-Hecke-paths}, this is a finite sum in untwisted affine ADE type. Beware that a number of dependencies are being suppressed from the notation to mitigate clutter: the number of folding times $N$, the folding times $(0=t_0 < t_1 < \cdots t_N < t_{N+1} = 1)$, and folding directions $(x_0,\cdots,x_N)$ depend on $\tau$, and the parahoric subgroups depend on $\lambda$.

By our above discussion, specifically Proposition \ref{proposition:existence-of-overline-x-k} and \eqref{eq:23}, we can realize each factor in \eqref{eq:58} as a specific Deodhar problem in the corresponding tangent root system. Therefore, \eqref{eq:58} gives an explicit (if perhaps complicated) combinatorial definition of $R^K_{\nu,\lambda}$.

\subsection{An $\widehat{SL_2}$ example}
\label{sec:an-sl-two-hat-example}
Let us compute $R^K_{\Lambda_0, \Lambda_0+\delta}$ for $\widehat{\SL_2}$. First we need to enumerate the set $\cH^{\Lambda_0+\delta,\Lambda_0}_{I_\infty}$ of $I_\infty$-Hecke paths of shape $\Lambda_0+\delta$ and endpoint $\Lambda_0$. By \eqref{eq:56}, each $I_\infty$-Hecke path will give us a chain in the Bruhat order from $\pi^{\Lambda_0+\delta}$ to $\pi^{\Lambda_0}$ in the Bruhat order. For the purposes of our calculations, it will be convenient to drop the irredundancy condition in \eqref{eq:56}; redundancies will be obvious, so we will not double count. To avoid confusion, we will write $y_0,y_1,\ldots$ (resp. $r_0,r_1,\ldots$) for the folding directions (resp. folding times) where we allow redunancies. Finally, we will restore the $x_0,x_1,\ldots$ (resp. $t_0,t_1,\ldots$) notation when we remove the redundancies.

Using the results of \cite{Muthiah-Orr-2019,Welch-2019} on the Bruhat order, we can explicitly compute the interval in the Bruhat order:
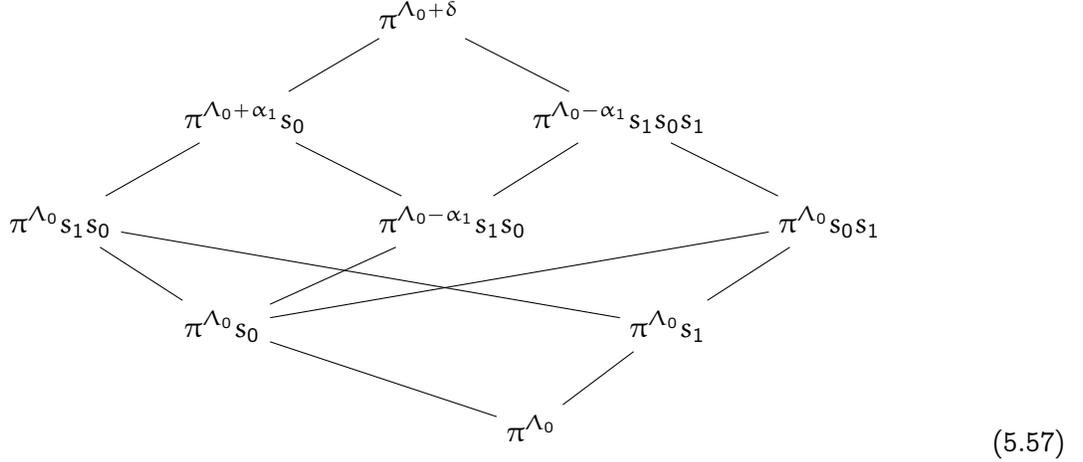
\begin{equation}
  \label{eq:137}
  \text{
  \begin{tikzpicture}[scale=.5]
    \node (A) at (0,0) {$\pi^{\Lambda_0 + \delta}$};
    \node [below left=1cm of A] (B) {$\pi^{\Lambda_0 + \alpha_1}s_0$};
    \node [below right=1cm of A] (C) {$\pi^{\Lambda_0 - \alpha_1}s_1s_0s_1$};
    \node [below left=1cm of B] (D) {$\pi^{\Lambda_0}s_1 s_0$};
    \node [below right=1cm of B] (E) {$\pi^{\Lambda_0 - \alpha_1} s_1s_0$};
    \node [below right=1cm of C] (F) {$\pi^{\Lambda_0}s_0 s_1$};
    \node [below right=1cm of D] (G) {$\pi^{\Lambda_0}s_0$};
    \node [below left=1cm of F] (H) {$\pi^{\Lambda_0}s_1$};
    \node [below left=1cm of H] (I) {$\pi^{\Lambda_0}$};
    \draw (A) -- (B);
    \draw (A) -- (C);
    \draw (B) -- (D);
    \draw (B) -- (E);
    \draw (C) -- (E);
    \draw (C) -- (F);
    \draw (D) -- (G);
    \draw (D) -- (H);
    \draw (E) -- (G);
    \draw (F) -- (G);
    \draw (F) -- (H);
    \draw (G) -- (I);
    \draw (H) -- (I);
  \end{tikzpicture}
}
\end{equation}
One can further verify that the above picture is also a picture of the Bruhat graph, i.e. every pair of elements that differ by a double affine reflection is pictured above.

We therefore count $6$ possible chains:
\begin{enumerate}
\item $\pi^{\Lambda_0 + \delta} > \pi^{\Lambda_0+\alpha_1} s_0 > \pi^{\Lambda_0} s_1 s_0 > \pi^{\Lambda_0}s_0 > \pi^{\Lambda_0}$
\item $\pi^{\Lambda_0 + \delta} > \pi^{\Lambda_0+\alpha_1} s_0 > \pi^{\Lambda_0} s_1 s_0 > \pi^{\Lambda_0}s_1 > \pi^{\Lambda_0}$
\item $\pi^{\Lambda_0 + \delta} > \pi^{\Lambda_0+\alpha_1} s_0 > \pi^{\Lambda_0-\alpha_1} s_1 s_0 > \pi^{\Lambda_0}s_0 > \pi^{\Lambda_0}$
\item $\pi^{\Lambda_0 + \delta} > \pi^{\Lambda_0-\alpha_1} s_1s_0s_1 > \pi^{\Lambda_0-\alpha_1} s_1 s_0 > \pi^{\Lambda_0}s_0 > \pi^{\Lambda_0}$
\item $\pi^{\Lambda_0 + \delta} > \pi^{\Lambda_0-\alpha_1} s_1s_0s_1 > \pi^{\Lambda_0} s_0 s_1 > \pi^{\Lambda_0}s_0 > \pi^{\Lambda_0}$
\item $\pi^{\Lambda_0 + \delta} > \pi^{\Lambda_0-\alpha_1} s_1 s_0 s_1 > \pi^{\Lambda_0} s_0 s_1 > \pi^{\Lambda_0}s_1 > \pi^{\Lambda_0}$
\end{enumerate}
Reading off the double affine reflections that give each step in the above chains (under left multiplication), we get sequences $\beta_1,\beta_2,\beta_3,\beta_4$ of double affine roots.
\begin{enumerate}
\item $\alpha_0,\alpha_1[1],\alpha_1,\alpha_0[1]$
\item $\alpha_0,\alpha_1[1],(\alpha_1+\delta)[1],\alpha_1$
\item $\alpha_0,\alpha_1,\alpha_1[-1],\alpha_0[1]$
\item $\alpha_1+\delta,\alpha_1+2\delta,\alpha_1[-1],\alpha_0[1]$  
\item $\alpha_1 + \delta, \alpha_1[-1], (\alpha_0+\delta)[2], \alpha_0[1]$
\item  $\alpha_1 + \delta, \alpha_1[-1], \alpha_0[1], \alpha_1$
\end{enumerate}

The folding directions are given by $y_0 = s_{\beta_1}, y_1 = s_{\beta_2}s_{\beta_1}, y_2 = s_{\beta_3}s_{\beta_2}s_{\beta_1}, y_3 = s_{\beta_4}s_{\beta_3}s_{\beta_2}s_{\beta_1}$. The folding times $r_0,r_1,r_2,r_3$ are constrained by the condition that $y_{k-1}^{-1}y_k \in W_{t_k}$.   The possible folding times are the following.
\begin{enumerate}
\item $0,\frac{1}{2},1,1$ 
\item $0, \frac{1}{2},1,1$
\item $0,0,\frac{1}{2},1$
\item $0,\bullet,\frac{1}{2},1$
\item $0,\frac{1}{2},\bullet,1$
\item $0,\frac{1}{2},1,\bullet$
\end{enumerate}
Here a $\bullet$ indicates that the condition $y_{k-1}^{-1}y_k \in W_{t_k}$ does not impose any constraint on the folding time $r_k$. These unconstrained folding times do not actually allow any additional freedom because no folding is happening at that time.

The folding times for each chain can be chosen so that $0 \leq r_0 \leq r_1 \leq r_2 \leq r_3 \leq 1$. Therefore, each chain above gives rise to an $I_\infty$-Hecke path. However, the same $I_\infty$-Hecke path arises multiple times. 
The above six chains give rise to only two distinct $I_\infty$-Hecke paths.

\subsubsection{The first $I_\infty$-Hecke path}
Chains (1), (2), and (4) give rise to the following $I_\infty$-Hecke path:
\begin{align}
  \label{eq:118}
 t_0 = 0, t_1 = \frac{1}{2} \\ 
 x_0 = s_0, x_1=\pi^{\Lambda_0}s_1s_0 \pi^{-\Lambda_0 - \delta}
\end{align}
This $I_\infty$-Hecke path contributes a factor of:
\begin{align}
  \label{eq:120}
 |I_\infty P_{[0,\eps)}/P_{[0,\eps)}| \cdot |x_0^{-1} I_\infty x_1 P_{[\frac{1}{2},\frac{1}{2}+\eps)}/P_{[\frac{1}{2},\frac{1}{2}+\eps)} \cap P_{(\frac{1}{2}-\eps,\frac{1}{2}]} P_{[\frac{1}{2},\frac{1}{2}+\eps)}/P_{[\frac{1}{2},\frac{1}{2}+\eps)}|
\end{align}
The first factor is a one dimensional Schubert cell in $P_{0}/P_{[0,\eps)}$. Explicitly, we can choose $i_0 = e_{-\alpha_0}(c_1)$ where $c_1 \in \kk$,  so it contributes a factor of $q$.
A short computations gives $\overline{x_0} = e$. For the second factor we compute $x_0^{-1}x_1 = \pi^{\alpha_0+\delta}s_{\alpha_0+\delta} = s_{-(\alpha_0+\delta)[1]}$. The simple negative roots in the tangent root system $\Phi_{\frac{1}{2}{(\Lambda_0+\delta)}}$ are $-\alpha_1$ and $-(\alpha_0+\delta)[1]$. Moreover, the reflection $s_{(\alpha_0+\delta)[1]}$ does not fix the segment germ $[\frac{1}{2},\frac{1}{2}+\eps)$. Therefore, the second factor is the intersection of a translate of a one-dimensional Schubert cell with the big cell in $P_{\frac{1}{2}}/P_{[\frac{1}{2},\frac{1}{2}+\eps)}$, so it contributes a factor of $q-1$. Explicitly, we can choose $i_1 = x_0 e_{-(\alpha_0+\delta)[1])}(c_2) x_0^{-1} = e_{-\alpha_1[1]}(c_2)$ where $c_2 \in \kk^\times$.
In terms of one-parameter Chevalley subgroups, the paths in the masure retracting onto the $I_\infty$-Hecke path correspond to the following subset of $I_\infty$:
\begin{align}
  \label{eq:121}
 \left\{ e_{-\alpha_0}(c_1) e_{-\alpha_1[1]}(c_2) \suchthat c_1 \in \kk, c_2 \in \kk^\times \right\}
\end{align}

\subsubsection{The second $I_\infty$-Hecke path}
Chains (3), (5), and (6) give rise to the following $I_\infty$-Hecke path:
\begin{align}
  \label{eq:119}
 t_0 = 0, t_1 = \frac{1}{2} \\ 
 x_0 = s_1s_0, x_1=\pi^{\Lambda_0}s_0 \pi^{-\Lambda_0 - \delta}  
\end{align}
The first factor is now a two dimensional Schubert cell, so it contributes a factor of $q^2$. Explicitly, one can choose $i_0 = e_{-\alpha_1}(c_1)e_{-\alpha_0}(c_2)$ where $c_1,c_2 \in \kk$.
One computes $\overline{x_0} = e$, and as in the first case $x_0^{-1}x_1 = \pi^{\alpha_0+\delta}s_{\alpha_0+\delta} = s_{(\alpha_0+\delta)[1]}$. We can choose $i_1 = x_0 e_{-(\alpha_0+\delta)[1])}(c_3) x_0^{-1} = e_{-\alpha_1[-1]}(c_2)$ where $c_2 \in \kk^\times$.
In terms of one-parameter Chevalley subgroups, the paths in the masure retracting onto the $I_\infty$-Hecke path correspond to the following subset of $I_\infty$:
\begin{align}
  \label{eq:123}
 \left\{ e_{-\alpha_1}(c_1)e_{-\alpha_0}(c_2) e_{-\alpha_1[-1]}(c_3)   \suchthat c_1,c_2 \in \kk, c_3 \in \kk^\times \right\}
\end{align}

\subsubsection{Weird phenomena about parabolic quotients}
\label{sec:weird-phenomena-about-parabolic-quotients}

Therefore, we have computed:
\begin{align}
  \label{eq:124}
R^{K}_{\Lambda_0,\Lambda_0+\delta}(q) = q^2(q-1) + q(q-1)
\end{align}
As we will discuss in \S \ref{sec:the-question-of-completions}, based on constructions of double affine Grassmannian slices, one na\"ively expects instead a polynomial of degree $4$. There we explain that this discrepancy should be due to the fact that we are working with the minimal Kac-Moody group as opposed to a completion.

This discrepancy also corresponds to the following unexpected behavior of the double affine Bruhat order. As we proved with Orr in \cite{Muthiah-Orr-2019}, the double affine Bruhat order on $\WTits$ (at least in untwisted affine ADE types) is graded by a length function $\ell : \WTits \rightarrow \ZZ$. Precisely, given $x \leq y$ the length of all maximal chains between $x$ and $y$ is equal to $\ell(y) - \ell(x)$. In particular, we expect that the length function should describe relative dimensions of double affine Schubert varieties in $G^+/I$.

One can also consider the parabolic quotient $\WTits/\Wv$. Generalizing definitions from the Coxeter case, one identifies $\WTits/\Wv$ with minimal length representatives and defines a Bruhat order by restriction to minimal length representatives.
Both $\pi^{\Lambda_0+\delta}$ and  $\pi^{\Lambda_0}$ are minimal length representatives modulo $\Wv$. Looking at \eqref{eq:137}, one computes that the interval between them in 
$\WTits/\Wv$ is exactly 
\begin{align}
  \label{eq:136}
  \pi^{\Lambda_0+\delta} \leftarrow \pi^{\Lambda_0+\alpha_1}s_0 \leftarrow \pi^{\Lambda_0-\alpha_1}s_1 s_0 \leftarrow \pi^{\Lambda_0}
\end{align}
which is only a $3$-term chain, despite the fact that $\ell(\pi^{\Lambda_0+\delta}) - \ell(\pi^{\Lambda_0}) = 4$. In particular, restricting the length function to minimal length representatives \emph{does not} give a grading of the poset $\WTits/\Wv$. 

The weird phenomenon is that $\pi^{\Lambda_0}$ has exactly two covers in the double affine Bruhat order: $\pi^{\Lambda_0}s_0$ and $\pi^{\Lambda_0}s_1$, and both are equivalent to $\pi^{\Lambda_0}$ modulo $\Wv$. This phenomenon does not occur for Coxeter groups because the length function is always a grading for parabolic quotients. In particular it does not occur when $\bG$ is finite type. This seems to indicate that double affine Schubert varieties in $G^+/K$ will have stranger behavior than those in $G^+/I$.

\section{Iwahori versions of $R$-polynomials }
\label{sec:iwahori-versions-of-R-polynomials}

\subsection{Iwahori-Spherical $R$-polynomials}
\label{sec:iwahori-spherical-r-polynomials}

Above we studied $(K \pi^\lambda K \cap I_\infty \pi^\nu K)/K$. Because we are considering $K \pi^\lambda K$ for $\lambda$ in the Tits cone, by the Cartan decomposition we assumed $\lambda$ dominant. The case where $\nu$ is dominant is the most interesting, but the above discussion also carries through for general $\nu$ in the Tits cone. Now we would like to replace $K \pi^\lambda K$ by $I \pi^\mu K$ to compute
\begin{align}
  \label{eq:128}
  (I \pi^\mu K \cap I_\infty \pi^\nu K)/K
\end{align}
where $\mu, \nu \in \overset{\circ}{\TitsCone}$ but not necessarily dominant. 

Let us define the set of {\it $I_\infty$-Hecke paths} of shape $\mu$ and endpoint $\nu$ by:
\begin{align}
  \label{eq:25}
\cH^{\mu,\nu}_{I_\infty} = \{ \rho_{I_\infty}(\phi) \suchthat \phi \in I.[0,\lambda] \textand \rho_{I_\infty}(\phi)(1)=\nu \}
\end{align}
That is, instead of the straight line path from $0$ to $\lambda$, we will consider the straight line path from $0$ to $\mu$. Making this modification will only cause a change at time $t_0=0$. Specifically, the condition that $g \in K$, is now replaced with the condition that $g \in I$. Then the set of choices for $i_0$ in \eqref{eq:018} is given by:
\begin{align}
  \label{eq:62}
 I_\infty x_0 \cdot P_{[0,0+ \epsilon)} /  P_{[0,0+ \epsilon)} \cap I \cdot P_{[0,0 + \epsilon)} /P_{[0,0+ \epsilon)} \subseteq P_{0}/P_{[0,\epsilon)}
\end{align}
We emphasize that here that $P_{[0,\epsilon)} = P_{[0,\epsilon \mu)}$. When $\mu$ is dominant, the set $I \cdot P_{[0,0 + \epsilon)} /P_{[0,0+ \epsilon)}$ is the big cell in $P_{0}/P_{[0,\epsilon)}$, but for general $\mu$ it is a finite-codimensional Schubert cell.

Let us analyze carefully, the condition for the intersection \eqref{eq:62} to be non-empty. Let $\mu_+$ be the dominant translate of $\mu$, and let $w \in \Wv$ be such that $\mu = w(\mu_+)$. Then we have a bijection between \eqref{eq:62} and:
\begin{align}
  \label{eq:64}
 I_\infty x_0 w^{-1} \cdot P_{[0,0+ \epsilon \mu_+)} /  P_{[0,0+ \epsilon\mu_+)} \cap I \cdot w^{-1} P_{[0,0 + \epsilon\mu_+)} /P_{[0,0+ \epsilon\mu_+)} \subseteq P_{0}/P_{[0,\epsilon\mu_+)}
\end{align}
This intersection is non-empty if and only if $w^{-1} \leq_{\Wv/\Wv_{\mu_+}} x_0 w^{-1}$, which explicitly means that there exists a sequence of positive roots $\beta_1, \ldots, \beta_s$ such that $\langle \beta_i , \mu_+ \rangle > 0$,
$x_0 = s_{\beta_N} \cdots s_{\beta_1}$, $w(\beta_1) > 0, ws_{\beta_1}(\beta_2) > 0, \cdots ws_{\beta_1} \cdots s_{\beta_{s-1}}(\beta_s) > 0$.
By a calculation analogous to the proof of Proposition \ref{prop:folding-condition-equiv-to-bruhat-order-condition}, we can conclude that this condition is equivalent to: 
\begin{align}
  \label{eq:68}
\pi^\mu \geq_{W_{t_0-1}/W_{[t_0-1,t_0-1+\eps)}} x_0 \pi^\mu  
\end{align}

Therefore, we see that the definition of $I_\infty$ Hecke paths will exactly as in the spherical case: folding times and folding directions that satisfy the inequalities \eqref{eq:48}. Precisely:
\begin{align}
  \label{eq:117}
  \pi^\mu \underset{W_{t_0-1}/W_{[t_0-1,t_0-1+\epsilon)},\Phi^+}{\geq} x_0\pi^\mu \underset{W_{t_1-1}/W_{[t_1-1,t_1-1+\epsilon},\Phi^+}{>} \cdots x_{N-1}\pi^\mu \underset{W_{t_N-1}/W_{[t_N-1,t_N-1+\epsilon)},\Phi^+}{>} x_N\pi^\mu
\end{align}
As before, we can deduce the finiteness of the set of $I_\infty$ Hecke paths from finiteness of intervals in the double affine Bruhat order. We therefore can make the following definition
\begin{Definition}
  \label{def:iwahori-spherical-R-poly}
Let $\mu,\nu \in \overset{\circ}{\cT}$. We define the \emph{Iwahori-spherical $R$-polynomial} by: 
\begin{align}
  \label{eq:58-iwahori}
  R_{\nu,\lambda} =  \sum_{\tau \in \cH^{\mu,\nu}_{I_\infty}} &| (I_\infty x_0 \cdot P_{[t_0,t_0+ \epsilon)} \cap I \cdot P_{[0,0 + \epsilon)} /P_{[0,0+ \epsilon)})/P_{[t_0,t_0+ \epsilon)}| \cdot \\ &\prod_{i=1}^N
  |(x_{k-1}^{-1} I_\infty x_k \cdot P_{[t_k,t_k+\epsilon)} \cap  P_{(t_{k}-\epsilon,t_k]} \cdot P_{[t_k,t_k+\epsilon)})/ P_{(t_{k}-\epsilon,t_k]}|
\end{align}
\end{Definition}
Again, we realize the factors in the product as explicit Deodhar problems, so we get a combinatorial formula.
\begin{Remark}
Observe that when $\lambda$ is dominant, the indexing sets in the formulas for $R_{\nu,\lambda}$ and $R^K_{\nu,\lambda}$ are both equal to $\cH^{\mu,\nu}_{I_\infty}$, but the terms in the summation are different. In particular, there is no conflict with the notation.
\end{Remark}

\subsection{Iwahori $R$-polynomials}

Let $\pi^\mu w, \pi^\nu v \in \WTits$ with $\mu,\nu \in \overset{\circ}{\cT}$. Then we can try to understand the intersection:
\begin{align}
  \label{eq:72}
 (I \pi^\mu w I \cap I_\infty \pi^\nu v I)/I 
\end{align}
As before, we start with $g \in I$ such that $g \pi^\mu w \in (I \pi^\mu w I \cap I_\infty \pi^\nu v I)/I$.

The computation we are considering is a refinement of the problem of computing $I \pi^\mu K \cap I_\infty \pi^\nu K$ considered in the previous section. First we recall how we understood that problem. Using the straight line path from $0$ to $\mu$, as \S \ref{sec:iwahori-spherical-r-polynomials}, we can find a sequence of folding times $0=t_0 < t_1 < \cdots < t_N < t_{N+1} = 1$ and folding directions $(x_0,\ldots,x_{N})$ such that for all $k \in \{0,\cdots,N\}$
\begin{align}
  \label{eq:69}
  g = i_0 \cdots i_k x_k p_k
\end{align}
where $i_0,\ldots,i_N \in I_\infty$, $x_k \in W$ and $p_k \in P_{[t_k,t_{k+1}]}$. The set of possible folding times and folding directions, $\cH^{\mu,\nu}_{I_\infty}$ is finite.

Now we need to further restrict our attention to \eqref{eq:72}. Let $\fc$ be the local chamber whose fixator is precisely $I$, and let $\Omega = \pi^\mu w(\fc)$. We have $P_\Omega = \pi^\mu w I w^{-1} \pi^{-\mu}$. We can then further decompose:
\begin{align}
  \label{eq:70}
  p_N = x_{N}^{-1} i_{N+1} x_{N+1} p_{N+1}
\end{align}
where $i_{N+1} \in I_{\infty}$, $x_{N}^{-1} x_{N+1} \in W_{1} = W_{\mu}$, and $p_{N+1} \in P_{\Omega}$. Observe that  $p_N \in P_{(1-\epsilon,1]}$. Observe also that we must have $x_{N+1}\pi^\mu w = \pi^\nu v$, so there is no choice for $x_{N+1}$. Fixing one of the finitely many possibilities for $x_N$, the set of possible choices for $i_{N+1}$ is in bijection with the intersection:
\begin{align}
  \label{eq:71}
 (x_{N}^{-1} I_\infty x_{N+1} P_{\Omega} \cap P_{(1-\epsilon,1]} P_{\Omega})/P_{\Omega}
\end{align}
As before, this can be realized as an explicit Deodhar problem. This is a bit different than previous Deodhar problems we have encountered because $P_{(1-\epsilon,1]} P_{\Omega}/P_{\Omega}$ is a translate of a finite dimensional Schubert cell, and $x_{N}^{-1} I_\infty x_{N+1} P_{\Omega}$ is a translate of a finite codimensional Schubert cell. The non-emptiness of \eqref{eq:71} imposes a further condition on $x_N$

We now define an $I_\infty$-Hecke path $\tau$ of shape $\pi^\mu w$ and endpoint $\pi^\nu v$ to be the data of folding times $0=t_0 < t_1 < \cdots < t_N < t_{N+1} = 1$ and folding directions $(x_0,\ldots,x_{N})$ inside $\cH^{\mu,\nu}_{I_\infty}$ such that
\begin{align}
  \label{eq:74}
  (x_{N}^{-1} I_\infty x_{N+1} P_{\Omega} \cap P_{(1-\epsilon,1]} P_{\Omega})/P_{\Omega} \neq \emptyset
\end{align}
where $x_{N+1} = \pi^{v}v w^{-1} \pi^{-\mu}$.
Let $\cH^{\pi^\mu w,\pi^\nu v}_{I_\infty}$ denote the set of (Iwahori) $I_\infty$-Hecke paths of shape $\pi^\mu w$ and endpoint $\pi^\nu v$. We see that $\cH^{\pi^\mu w,\pi^\nu v}_{I_\infty}$ is a subset of $\cH^{\mu,\nu}_{I_\infty}$; in particular, it is a finite set. Now we can define $R$-polynomials.
\begin{Definition}
  \label{def:iwahori-iwahori-R-poly}
Let $\pi^\mu w, \pi^\nu v \in \WTits$ with $\mu,\nu \in \overset{\circ}{\cT}$.
  We define the \emph{(Iwahori) $R$-polynomial} by:
\begin{align}
  \label{eq:58-iwahori-iwahori}
  R_{\pi^\mu w, \pi^\nu v} =  \sum_{\tau \in \cH^{\pi^\mu w,\pi^\nu v}_{I_\infty}} &| (I_\infty x_0 \cdot P_{[0,0+ \epsilon)} \cap I \cdot P_{[0,0 + \epsilon)} )/P_{[0,0+ \epsilon)}| \cdot \\ &\prod_{i=1}^N
  |(x_{k-1}^{-1} I_\infty x_k \cdot P_{[t_k,t_k+\epsilon)} \cap  P_{(t_{k}-\epsilon,t_k]} \cdot P_{[t_k,t_k+\epsilon)})/ P_{[t_{k},t_k+\epsilon)}| \cdot \\ & |(x_{N}^{-1} I_\infty x_{N+1} P_{\Omega} \cap P_{(1-\epsilon,1]} P_{\Omega})/P_{\Omega}|
\end{align}
\end{Definition}
Here, as before, the dependence of the folding times and directions on $\tau$ is notationally suppressed. The subscripts of the parahoric subgroups refer to the straight line path from $0$ to $\mu$, and $P_{\Omega} = \pi^\mu w I w^{-1} \pi^{-\mu}$.

\section{Kazhdan-Lusztig polynomials}
\label{sec:kazhdan-lusztig-polynomials}
\newcommand{\IC}{\mathrm{IC}}
\newcommand{\Tr}{\mathrm{Tr}}

Our next goal is to pass from our Kazhdan-Lusztig $R$-polynomials, which are point-counts, to Kazhdan-Lusztig $P$-polynomials, which are Poincar\'e polynomials of intersection cohomology stalks. The method we will use is exactly the method used by Kazhdan and Lusztig in their original paper on intersection cohomology of Schubert varieties \cite{Kazhdan-Lusztig-1980}. Their method involves non-trivial geometry: the Grothendieck-Lefschetz fixed point theorem, Deligne's theory of weights, and a theorem of Springer about stalks of $\GG_m$-equivariant sheaves on cones. The output however is a purely combinatorial recursive formula for calculating the $P$-polynomials. We recall their method briefly below.

\subsection{The Kazhdan-Lusztig recursion formula}
\label{sec:the-kazhdan-lusztig-recursion-formula}
In this subsection, as in \S \ref{sec:deodhar-problems-and-twinned-buildings}, we will consider a Kac-Moody $\G$ with opposite Borels $\B$ and $\B^-$ considered over $\kk$. Fix elements $v,w$ in the Weyl group with $v \leq w$, and $X^w_v$ denote the transversal slice to the Schubert variety $\overline{\B v \B/\B}$  embedded as a sub variety of the Schubert variety $\overline{\B w \B/\B}$. We consider these objects as varieties over $\kk$. Explicitly, we have:
\begin{align}
  \label{eq:73}
  X^w_v = \overline{\B w \B/\B} \cap \B^- v \B/\B
\end{align}
For every $u$, we also denote:
\begin{align}
  \label{eq:93}
  \overset{\circ}{X^u_v} = \B u \B/\B \cap \B^- v \B/\B
\end{align}
We have a stratification:
\begin{align}
  \label{eq:94}
  X^w_v = \bigsqcup_{u: v \leq u \leq w} \overset{\circ}{X^u_v}
\end{align}

For any irreducible variety, let $\IC(X)$ denote the $\ell$-adic intersection cohomology sheaf of $X$ (for $\ell \neq \text{char}(\kk)$), normalized so that stalks are concentrated in non-negative degrees. Note that under this normalization, Verdier duality causes $\IC(X)$ to be shifted by $2 \dim(X)$ in cohomological degree.

Let $P_{u,w}$ denote the trace of Frobenius acting on the stalk of $\IC(X^w_u)$ at the point $u$. By Springer's theorem \cite{Springer}, the stalk at $u$ with its Frobenius action is canonically isomorphic to the total intersection cohomology $\HH^\bullet(IC(X^w_u))$ of $X^w_u$. Because the stratum  $\overset{\circ}{X^u_v}$ of $X^w_v$ has a transversal slice isomorphic to $X^w_u$, $P_{w,u}$ is also equal to the trace of Frobenius acting on the stalk of $\IC(X^w_v)$ at any point of $\overset{\circ}{X^u_v}$. Therefore, by the Grothendieck-Lefschetz trace formula, we have
\begin{align}
  \label{eq:95}
  \Tr( \HH_c^\bullet(\IC(X^w_v))) = \sum_{u: v \leq u \leq w} |\overset{\circ}{X^u_v}(\kk)| \cdot P_{w,u}
\end{align}
where we have written $\Tr(\cdot)$ to denote the (super)-trace of Frobenius acting on a complex.

In the course of their computation, Kazhdan and Lusztig prove that $P_{w,u}$ is in fact a polynomial in $q$, that the degree of $P_{u,w}$ is less than or equal to $\frac{1}{2}(\ell(w) - \ell(u) - 1)$ for $u < w$, and $P_{w,w} = 1$.  Finally, they show that $P_{u,w}$ is in fact the Poincar\'e polynomial of $\IC(X^w_u)$ at the point $u$. This final conclusion uses the strength of Deligne's theory of weights, and the facts that intersection cohomology of projective varieties is pure and satisfies the Hard Lefschetz theorem. This implies the deep fact that $P_{u,w}$ has positive coefficients.

By Springer's theorem, the compactly supported hypercohomology $\HH_c^\bullet(\IC(X^w_v))$ is canonically isomorphic to the $!$-stalk of $\IC(X^w_v)$ at the point $v$. From this, one computes $\Tr( \HH_c^\bullet(\IC(X^w_v))) = q^{\ell(w)-\ell(y)}\overline{P_{v,w}}$, where $\overline{P_{v,w}}(q) = P_{v,w}(q^{-1})$.

Therefore, formula \eqref{eq:95} can be rewritten equivalently as the familiar recursion for $R$- and $P$-polynomials:
\begin{align}
  \label{eq:96}
  q^{\ell(w)-\ell(v)}\overline{P_{v,w}} - P_{v,w} = \sum_{u: v < u \leq w} R_{v,u} \cdot P_{u,w}
\end{align}
The $R$-polynomials are defined by $R_{v,u}(q) = |\overset{\circ}{X^u_v}(\kk)|$.
The degree bounds imply that there is no cancellation in the left hand side, so \eqref{eq:96} characterizes the Kazhdan-Lusztig $P$-polynomials. Furthermore, the right hand side in

\subsection{The Iwahori case} 
\label{sec:the-iwahori-case}

A close look at the calculation above, we see the essential ingredients are: (1) the Bruhat order that controls the strata that appear in transversal slices, (2) the length function that controls the codimension of strata, and (3) the $R$-polynomials that give the point-counts of strata over finite fields. In the double affine case, the author and Daniel Orr have studied the Bruhat order and length function in detail \cite{Muthiah-2018,Muthiah-Orr-2019}, and the purpose of the present article is to give a definition of $R$-polynomials.

\begin{Conjecture}
  \label{conj:existence-of-transversal-slices-to-schubert-varieties}
  For every $\pi^\mu w, \pi^\nu v \in \WTits$, there exists an affine algebraic variety (defined over $\kk$), with contracting $\GG_m$-action $X^{\pi^\mu w}_{\pi^\nu v}$, whose dimension is $\ell(\pi^\mu w) - \ell(\pi^\nu v)$ and admitting a stratification by open sets
  \begin{align}
    \label{eq:97}
    X^{\pi^\mu w}_{\pi^\nu v} = \bigsqcup_{\pi^\xi u : \pi^\nu v \leq \pi^\xi u \leq \pi^\mu w} \overset{\circ}{X^{\pi^\xi u}_{\pi^\nu v}}
  \end{align}
such that 
\begin{align}
  \label{eq:98}
\overset{\circ}{X^{\pi^\xi u}_{\pi^\nu v}}(\kk) \cong (I \pi^\xi u I \cap I_\infty \pi^\nu v I )/I 
\end{align}
as sets.
Furthermore, each stratum $\overset{\circ}{X^{\pi^\xi u}_{\pi^\nu v}}$ admits a transversal slice isomorphic to $X^{\pi^\mu w}_{\pi^\nu v}$
\end{Conjecture}
The affine variety $X^{\pi^\mu w}_{\pi^\nu v}$  provides a partial compactification of the set $\overset{\circ}{X^{\pi^\mu w}_{\pi^\nu v}}$, which has a group theoretic nature.

Assuming this conjecture, and assuming that our definition of $R$-polynomials correctly computes the cardinality of $(I \pi^\xi u I \cap I_\infty \pi^\nu v I )/I$, we immediately can apply the Kazhdan-Lusztig argument. As before, one can define $P_{\pi^\nu v, \pi^\mu w}$ to be the Poincar\'e polynomials of the intersection cohomology stalks of
$X^{\pi^\mu w}_{\pi^\nu v}$ at the point $\pi^\nu v$. Then we have
\begin{align}
  \label{eq:99}
  \deg P_{\pi^\nu v, \pi^\mu w} \leq \frac{1}{2} ( \ell(\pi^\mu w) - \ell(\pi^\nu v) - 1)
\end{align}
for all $\pi^\nu v$ and $\pi^\mu w$ with $\pi^\nu v < \pi^\mu w$, and:
\begin{align}
  \label{eq:100}
  P_{\pi^\mu w, \pi^\mu w} = 1
\end{align}
Finally, we have the recursion
\begin{align}
  \label{eq:101}
  q^{\ell(\pi^\mu w)-\ell(\pi^\nu v)}\overline{P_{\pi^\nu v,\pi^\mu w}} - P_{\pi^\nu v,\pi^\mu w} = \sum_{\pi^\xi u: \pi^\nu v < \pi^\xi u \leq \pi^\mu w} R_{\pi^\nu v,\pi^\xi u} \cdot P_{\pi^\xi u,\pi^\mu w}
\end{align}
that together with \eqref{eq:99} and \eqref{eq:100} uniquely characterizes $P_{\pi^\nu v,\pi^\mu w}$. Observe that once we have the existence of the variety of Conjecture \ref{conj:existence-of-transversal-slices-to-schubert-varieties}, the Kazhdan-Lusztig argument automatically would tell us that the coefficients of $P_{\pi^\nu v,\pi^\mu w}$ are positive.

For now, because we do not presently have Conjecture \ref{conj:existence-of-transversal-slices-to-schubert-varieties} available, we instead take \eqref{eq:101} as a definition. Note, however, that this formula may not have a solution because it requires that the coefficients of right hand side of \eqref{eq:101}, as a polynomial in $q$, must satisfy a signed palindrome property. A failure of such a property would indicate that Conjecture \ref{conj:existence-of-transversal-slices-to-schubert-varieties} would need to be modified. Such a modification would not be unthinkable because, as we will explain below, a modification is expected in the spherical case.

\subsubsection{$R$-polynomials and the  matrix of Kazhdan-Lusztig involution}
\label{sec:R-polynomials-and-the-matrix-of Kazhdan-Lusztig-involution}

The classical development of Kazhdan-Lusztig polynomials typically proceeds using the Hecke algebra and the Kazhdan-Lusztig involution.  Our approach described above currently makes no use of the Hecke algebra, but nonetheless one expects that there should be a close relationship with the Iwahori-Hecke algebra of $G$. 

Recall the Iwahori-Hecke algebra $\cH_I$, which consists of $I$-biinvariant complex valued functions on $G^+$ that are supported on finitely many double cosets \cite{Braverman-Kazhdan-Patnaik,BardyPanse-Gaussent-Rousseau-2016}. The double coset decomposition of $G^+$ into $I$-double cosets gives $\cH_I$ the basis $\{T_x\}$ indexed by $x \in W_{\cT}$.  By \cite{Muthiah-2018, BardyPanse-Gaussent-Rousseau-2016}, the structure coefficients of this basis lie in $\ZZ[q,q^{-1}]$. Therefore, we can redefine $\cH_I$ to be the $\ZZ[q,q^{-1}]$-span of the basis $\{T_x\}$.

Define the algebra involution $\overline{\cdot} : \ZZ[q,q^{-1}] \rightarrow \ZZ[q,q^{-1}]$  by $\overline{q} = q^{-1}$. In the special case where $\bG$ is finite type, the algebra $\cH_I$ is isomorphic to the affine Hecke algebra and has a Kazhdan-Lusztig involution $\overline{\cdot} : \cH_I \rightarrow \cH_I$, which is defined on the $\{T_x\}$ basis by 
\begin{align}
  \label{eq:135}
  \overline{T_{\pi^\mu w}} = \sum_{\pi^\nu v \leq \pi^\mu w} \overline{R_{\pi^\nu v, \pi^\mu w}} q^{-\ell(\pi^\nu v)} T_{\pi^\nu v}
\end{align}
and for arbitrary elements by the requirement that it be semilinear with respect to the involution $\overline{\cdot}$ of $\ZZ[q,q^{-1}]$.

In the case where $\bG$ is untwisted affine Kac-Moody type, we do not currently have a definition of the Kazhdan-Lusztig involution, but we can now make sense of the right hand side of \eqref{eq:135}. The definition of the $R$-polynomial is the subject the present paper, the double affine Bruhat order was defined and studied in \cite{Braverman-Kazhdan-Patnaik,Muthiah-2018}, and the definition of the length function $\ell(\cdot)$ was defined and studied in \cite{Muthiah-2018,Muthiah-Orr-2019}. However, the right hand side of \eqref{eq:135} has infinitely many terms. We therefore conjecture that there exist some completion of $\cH_I$ such that the right hand side of \eqref{eq:135} defines an algebra involution. We hope that such a result should also help clarify the complicated combinatorics of $I_\infty$-Hecke paths.

\subsection{Affine Grassmannian slices and Coulomb branches}
\label{sec:affine-Grassmannian-slices-and-Coulomb-branches}

Let us temporarily consider the case when $\bG$ is finite-type. In this case,
one can identify $G/K$ with the $\kk$-points of certain ind-scheme called the \emph{affine Grassmannian}. Let $\lambda, \nu$ be dominant coweights such that $\lambda \geq \nu$. The intersection $(K \pi^\lambda K \cap I_\infty \pi^\nu K)/K$ is a locally closed subvariety of $G/K$ of dimension $\langle 2 \rho, \lambda -\nu \rangle$. The \emph{affine Grassmannian slice} $\overline{\cW}^\lambda_\nu$ is defined to be the one-sided closure:
\begin{align}
  \label{eq:125}
\overline{\cW}^\lambda_\nu = (\overline{K \pi^\lambda K} \cap I_\infty \pi^\nu K)/K = \bigsqcup_{\mu : \lambda \geq \mu \geq \nu} (K \pi^\mu K \cap I_\infty \pi^\nu K)/K 
\end{align}
The variety $\overline{\cW}^\lambda_\nu$ is a an example of a Schubert slice, and is in particular a conical affine varieties. Among Schubert slices, affine Grassmannian slices are quite special and have many special properties and structures. 

\subsubsection{Geometric Satake and the Kato-Lusztig formula}
One does not expect explicit formulas for Kazhdan-Lusztig polynomials in general, but a special feature of affine Grassmannian slices is the following explicit formula for its intersection homology.

Let $P^{K}_{\nu,\lambda}$ be the Poincar\'e polynomial of the intersection homology of $\overline{\cW}^\lambda_\nu$. The Kato-Lusztig formula (\cite{Kato,Lusztig}) says
\begin{align}
  \label{eq:126}
 P^K_{\nu,\lambda}(q) = K_{\nu,\lambda}(q) 
\end{align}
where $K_{\nu,\lambda}(q)$ is the so called $q$-analogue of weight multiplicity. In brief, $K_{\nu,\lambda}(q)$ are the matrix entries of the change of basis matrix between the Hall-Littlewood polynomials and Weyl characters. We mention Viswanath has developed the theory of Hall-Littlewood polynomials in the symmetrizable Kac-Moody case \cite{Viswanath}.

Formula \eqref{eq:126} is considered to have been the first glimpse of the \emph{geometric Satake correspondence}, and we expect its generalization to be a key milestone for the Kac-Moody geometric Satake correspondence.

\subsubsection{Coulomb branch construction of Kac-Moody affine Grassmannian slices}

Recently, Braverman, Finkelberg, and Nakajima have given another construction of affine Grassmannian slices by giving a mathematically rigorous construction of the notion of a Coulomb branch. Precisely speaking, they construct the Coulomb branches of $3d$ $\cN = 4$ gauge theories of cotangent type, a specific class of which is the so called quiver gauge theories. 

Let us now drop our restriction that $\bG$ is finite type, but assume that $\bG$ has symmetric Cartan matrix. The Braverman-Finkelberg-Nakajima Coulomb branch construction produces an affine variety $\overline{\cW}^\lambda_\nu$ of dimension $\langle 2 \rho, \lambda - \nu \rangle$ associated to a pair of dominant coweights $\lambda$ and $\nu$.
In the case that $\bG$ is finite-type, they prove that this is exactly the affine Grassmannian slice considered above.
Therefore, one expects that in Kac-Moody type, $\overline{\cW}^\lambda_\nu$ should give a good definition of \emph{Kac-Moody affine Grassmannian slices}. In affine type, these are usually called \emph{double affine Grassmannian slices}. 

Recent work by Nakajima and Weekes \cite{Nakajima-Weekes} has further extended the construction to many symmetrizable Kac-Moody types allowing one to relax the symmetric condition.

\subsubsection{The question of completions}
\label{sec:the-question-of-completions}

This variety contains a dense open subset $\cW^\lambda_\nu$ that in finite-type is in bijection with  $(K \pi^\lambda  K \cap I_\infty \pi^\nu K)/K$. However, in finite-type we do not expect this is true without considering an appropriate completion of the Kac-Moody group. For example, when $\bG = \widehat{\SL_2}$, $\cW^{\Lambda_0+\delta}_{\Lambda_0}$ is known to be $4$-dimensional, but our calculation in \S \ref{sec:an-sl-two-hat-example} shows that $(K \pi^{\Lambda_0+\delta}  K \cap I_\infty \pi^{\Lambda_0} K)/K$ behaves like a $3$-dimensional variety. We believe this discrepancy is because we are considering only the minimal Kac-Moody group. As we explained in \ref{sec:an-sl-two-hat-example}, this example also exhibits some weird behavior of the parabolic quotient $\WTits/\Wv$.

In finite-type, it is known that coordinate ring of $\overline{\cW}^\lambda_\nu$ quantizes to a non-commutative algebra called the truncated shifted Yangians. One can define the truncated shifted Yangians in Kac-Moody type as well, and they most naturally behave like functions on subsets of a certain completion of the Kac-Moody group. Specifically, they naturally quantize the ``scattering matrix'' model of affine Grassmannian slices \cite[\S 2(xi) ]{Braverman-Finkelberg-Nakajima}, which at first appears to require completions for both the positive and negative parts of the Kac-Moody group. See also \cite{Krylov,Muthiah-Weekes} for further discussion of the scattering matrix model.

\subsubsection{Strange stratifications and correction factors}

In finite-type, one has a stratification
\begin{align}
  \label{eq:127}
  \overline{\cW}^\lambda_\nu = \bigsqcup_{\mu : \lambda \geq \mu \geq \nu} \cW^\mu_\nu
\end{align}
This na\"ive stratification does not work in Kac-Moody type. In affine type A, Nakajima and Takayama \cite[Theorem 7.26]{Nakajima-Takayama} have explicitly described the correct stratification of Coulomb branches. We refer the reader to their paper for the details, but we mention that symmetric powers of algebraic surfaces naturally appear as extra factors in their stratification

Similar strange stratifications also appear for \emph{affine Zastava spaces}. As explained in \cite{Braverman-Finkelberg-Kazhdan}, the strange stratification of affine Zastava spaces is exactly responsible for a correction factor in the affine Gindikin-Karpelevich formula. This correction factor also appears in the computation of the Satake Transform \cite{Braverman-Kazhdan-Patnaik,BardyPanse-Gaussent-Rousseau-2019}. Finally, it appears as a correction factor in an identity for the Poincar\'e series for the Weyl group $\Wv$ \cite{Macdonald,Muthiah-Puskas-Whitehead}. It would be very interesting to explain these strange stratifications in terms of $p$-adic Kac-Moody groups!

\bibliographystyle{amsalpha}
\bibliography{references}

\end{document}